\documentclass[a4paper,UKenglish,cleveref, autoref, thm-restate]{lipics-v2021}
\pdfoutput=1
\title{Monoidal Width: Unifying Tree Width, Path Width and Branch Width}
\author{Elena {Di Lavore}}{Tallinn University of Technology, Estonia}{elendi@ttu.ee}{https://orcid.org/0000-0002-7783-5079}{}
\author{{Pawe\l} Soboci\'nski}{Tallinn University of Technology, Estonia}{pawel.sobocinski@ttu.ee}{https://orcid.org/0000-0002-7992-968}{}
\authorrunning{E. Di Lavore and P. Soboci\'nski}
\Copyright{Elena Di Lavore and {Pawe\l} Soboci\'nski}
\nolinenumbers

\usepackage{default}
\usepackage{wrapfig}
\begin{document}
\maketitle
\begin{abstract}
  We introduce \emph{monoidal width} as a measure of the difficulty of decomposing morphisms in monoidal categories.
  For graphs, we show that monoidal width and two variations capture existing notions, namely branch width,
  tree width and path width.
  We propose that monoidal width: (i) is a promising concept that, while capturing known measures, can similarly be instantiated in other settings, avoiding the need for ad-hoc domain-specific definitions and (ii) comes with a general, formal algebraic notion of \emph{decomposition} using the language of monoidal categories.
\end{abstract}

\section{Introduction}\label{sec:intro}
\setlength{\intextsep}{0pt}
\begin{wrapfigure}{r}{0.4\textwidth}
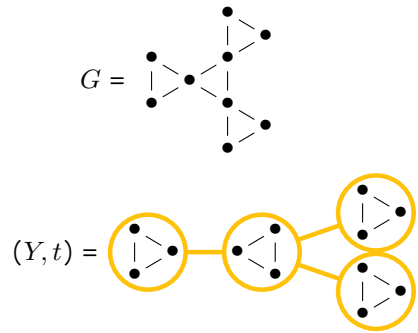

  \centering
  \treeDecExFig{}
  \caption{The cost of a tree decomposition \((Y,t)\) of a graph \(G\) records the cost of decomposing \(G\) into a tree shape. The cost of \((Y,t)\) is \(3\) as its biggest component has three vertices.}\label{fig:tree-dec}
\end{wrapfigure}
Tree width is a measure of complexity for graphs that was independently defined by different authors~\cite{bertele1973treewidth,halin1976treewidth,robertson1986graph-minorsII}. Every nonempty graph has a tree width, which is a positive integer.
The interest in this concept is partly due to its algorithmic properties. For example, important
problems that are NP-hard on generic graphs have linear time algorithms on graphs with bounded tree width~\cite{bodlaender1992tourist,bodlaender2008combinatorial,courcelle1990monadic}.
Similar motivations lead to the definitions of other notions of complexity for graphs such as path width~\cite{robertson1983graph-minorsI}, branch width~\cite{robertson1991graph-minorsX}, rank width~\cite{oum2006rank-width}, clique width~\cite{courcelle2000upper-clique} and cut width~\cite{adolphson1973optimal,aharonov2006quantum,chudnovsky2011well}. All of them share a similar basic idea: in each case, a specific notion of legal decomposition is priced according to the most expensive operation involved, and the price of the cheapest decomposition is the width.
We will generically refer to tree width, path with, branch width etc.\ as \emph{graph widths}.

Graph widths record the cost of decomposing a graph according to some decomposition rules.
For instance, tree width records---roughly speaking---how costly it is to decompose a graph in a tree shape, where the cost is given by the maximum number of vertices of the components of a decomposition. An example decomposition is  in~\Cref{fig:tree-dec}.

The algebra of monoidal categories can be seen as a general process algebra of concurrent processes, featuring sequential and parallel composition.
In recent years it has been used to describe artefacts of a computational nature; e.g.\ Petri nets~\cite{fong2018seven}, quantum circuits~\cite{Coecke2017,DuncanKPW20}, signal flow graphs~\cite{fong2018seven,Bonchi0Z21}, electrical circuits~\cite{Comfort2021,Boisseau2021}, digital circuits~\cite{GhicaJL17}, stochastic processes~\cite{fritz2020,cho2019} and games~\cite{GhaniHWZ18} to name just a few.  However, while semantics of computational artefacts is often compatible with the algebra of categories, \emph{performance} sometimes is not.
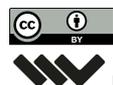
\setlength{\intextsep}{0pt}
\begin{wrapfigure}{l}{0.4\textwidth}
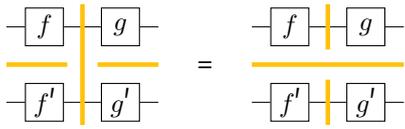

  \centering
  \interchangelawDecFig{}
  \caption{This morphism can be decomposed in two different ways: \((f \tensor f') \dcomp (g \tensor g') = (f \dcomp g) \tensor (f' \dcomp g')\). The left one is more costly as we need to synchronize the two processes on the common boundary when composing them.}\label{fig:interchange-law}
\end{wrapfigure}
Consider the interchange law, illustrated in~\Cref{fig:interchange-law}. If, as is usually the case, the boxes represent some kind of processes, then composing along a common boundary involves synchronisation, coordination or resource sharing. Then to compute the composite system efficiently, it is usually a good idea to minimise the size of the boundary along which one composes. An example is compositional reachability checking in Petri nets of Rathke et.\ al.~\cite{rathke2014compositional}: calculating the sequential composition is exponential in the size of the boundary. On the other hand, the monoidal product is usually cheap since---as indicated by the wiring in the string diagrams---there is no information sharing between the components. In other words, the right hand side of~\Cref{fig:interchange-law} is a more efficient way to compute: performance does not respect the middle-four interchange law of monoidal categories!


Our main contribution is to borrow the ideas behind graph widths to make this idea precise, and to measure the complexity of morphisms in monoidal categories. We introduce the concept of \emph{monoidal width} and two variants.

\medskip
General approaches have the potential to be of wider use.
For example, the various notions of graph decompositions, while clearly similar~\cite{pudlak1988graph-complexity,courcelle1990monadic}, are quite concrete and dependent on the underlying graph model. For example, tree width is traditionally defined for undirected graphs, and the seemingly mild generalisation to directed graphs has already resulted in several works~\cite{johnson2001directed,berwanger2012dag,hunter2008digraph,safari2005d}. A more general approach helps to clarify the research landscape and inform appropriate instantiations for specific models of interest. Second, the \emph{optimal decomposition} itself is a valuable piece of data that is discarded when talking about width as a mere number. As mentioned previously, decompositions in the literature are defined specifically for individual graph models and while they carry the intuition of obtaining composite graphs from simpler components, an \emph{explicit algebra of composition} is often missing. With category theory in the picture, we shift the focus from a number to formal, executable expressions that describe optimal decompositions.

\textbf{Contributions.}
We introduce \emph{monoidal width}, following the idea of the cost of decomposing a morphism into compositions and monoidal products of chosen atomic morphisms.
Monoidal width, and restricted versions of it, are 
instantiated to the category of cospans of graphs to recover the usual notions of branch width, tree width and path width.
These results build a bridge between the algebraic and the combinatorial approaches to graphs.

\textbf{Structure of the paper.}
In~\Cref{sec:monoidal-widths} we define monoidal width and its versions restricted to tree and path shapes.
The definitions (as in~\cite{robertson1983graph-minorsI,robertson1986graph-minorsII,robertson1991graph-minorsX}) of tree width, path width and branch width are recalled in~\Cref{sec:graphs} and given alternative recursive characterisations.
We show, in~\Cref{sec:mwd-graph-widths}, that these correspond to branch width, tree width and path width, respectively, when instantiated in the category of cospans of graphs, introduced in~\Cref{sec:cospan-graph}.

\textbf{Related work.}
The work of Pudl{\'a}k, R{\"o}dl and Savick{\`y}~\cite{pudlak1988graph-complexity}
addresses the complexity of graphs in a syntactical way: the authors define the complexity of a graph to be the minimum number of operations needed to define a graph.
 Bauderon and Courcelle~\cite{bauderon1987graph} follow a similar idea and define a language to construct graphs from given generators. In particular, the cost of a decomposition is measured by counting \emph{shared names}, which is clearly closely related to penalising sequential composition as in monoidal width. Nevertheless, these approaches are specific to particular, concrete notions of graphs, whereas our work concerns the more general algebraic framework of monoidal categories.

Abstract approaches to width have received some attention recently, with a number of diverse contributions.
Blume et.\ al.~\cite{blume2011treewidth}, similarly to our work,
 use (the category of) cospans of graphs as a formal setting to study graph decompositions: indeed, a major insight of loc.\ cit.\ is that tree decompositions are tree-shaped diagrams in the cospan category, and the original graph is reconstructed as a colimit of such a diagram. Our approach is more general, however, emphasising the relevance of the algebra of monoidal categories, of which cospan categories are just one family of examples.
Abramsky et.\ al.~\cite{feder1998computational} give a coalgebraic characterization of  tree width of relational structures (and graphs in particular).
Bumpus and Kocsis~\cite{bumpus2021spined} also generalise tree width to the categorical setting, although their approach is technically far removed from ours: they generalise tree width to be a functor satisfying some properties, relying on characterisation of tree width in terms of Halin's \(S\)-functions~\cite{halin1976treewidth}.
\setlength{\intextsep}{0pt}
\begin{wrapfigure}{r}{0.45\textwidth}
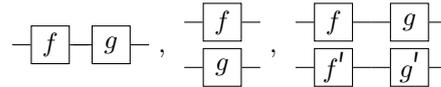

  \centering
  \sequentialFig{}, \parallelFig{}, \interchangelawFig{}
  \setlength{\belowcaptionskip}{0pt}
  \caption{String diagrammatic notation.}\label{fig:string-diagrams}
\end{wrapfigure}
\textbf{Preliminaries.}
We use string diagrams~\cite{joyal1991geometry,selinger2010survey}:
sequential and parallel composition of \(f\) and \(g\) is drawn as in \Cref{fig:string-diagrams}, left and middle, respectively.
Much of the bureaucracy, e.g. the interchange law \((f \dcomp g) \tensor (f' \dcomp g') = (f \tensor f') \dcomp (g \tensor g')\), disappears (\Cref{fig:string-diagrams}, right).
\emph{Props}~\cite{MacLane1965,Lack2004a} are important examples of monoidal categories.
They are symmetric strict monoidal, with natural numbers as objects, and addition as monoidal product on objects. Roughly speaking, morphisms can be thought of as processes, and the objects (natural numbers) keep track of the number of inputs or outputs of a process.


\section{Monoidal widths}\label{sec:monoidal-widths}
In this section we introduce the central original concepts of the paper:
 \emph{monoidal width}
   and
two variations called \emph{monoidal tree width} and \emph{monoidal path width}.
Each has a notion of \emph{decomposition}:
given a morphism $f$, a decomposition is a tree with internal nodes labelled with operations $\{\,\dcomp\,,\otimes\}$ of monoidal categories, and leaves labelled with atomic morphisms. Evaluating a valid decomposition yields $f$.
But, in general, $f$ can be decomposed in different ways.
The \emph{width} is the cost of the ``cheapest'' decomposition.
The basic idea of ``paying a price'' for performing an operation is captured by the following:
\begin{definition}
  Let \(\cat{C}\) be a monoidal category and let \(\decgenerators\) be a set of morphisms in \(\cat{C}\), which we shall refer to as \emph{atomic}.
  A \emph{weight function} for \((\cat{C},\decgenerators)\) is a function \(\nodeweight \colon \decgenerators \union \monoidaloperations{\cat{C}} \to \naturals\) such that: (i) \(\nodeweight(X \tensor Y) = \nodeweight(X) + \nodeweight(Y)\), and (ii) \(\nodeweight(\tensor) = 0\).
\end{definition}
If $\cat{C}$ is a prop then typically we let $\nodeweight(1)\defn 1$.
We do not assume anything about the structure of atomic morphisms in $\decgenerators$; they merely do not necessarily need to be decomposed further.

We shall consider three kinds of decomposition of morphisms in monoidal categories.
In the first kind, \emph{monoidal decomposition}, sequential composition and monoidal product can be used without restriction. Monoidal decompositions, when instantiated in the category of cospans of graphs, correspond to branch decompositions, as shown in~\Cref{sec:mwd-branch-width}.

\begin{definition}[Monoidal decomposition]\label{defn:monoidalDecomposition}
  Let \(\cat{C}\) be a monoidal category and \(\decgenerators\) be a set of morphisms.
  The set $\decset{f}$ of \emph{monoidal decompositions} of \(f \colon A \to B\) in $\cat{C}$ is defined recursively:
  \begin{align*}
  \decset{f} \quad \Coloneqq \quad & \leafgenerator{f} &\text{if } f \in \decgenerators \\
                      \mid \quad & \nodegenerator{d_{1}}{\tensor}{d_{2}} &\text{if } d_1\in \decset{f_1},\,d_2\in \decset{f_2} \text{ and } f =_\cat{C} f_1\tensor f_2 \\
                      \mid \quad & \nodegenerator{d_{1}}{\dcomp_{X}}{d_{2}} &\text{if }d_1\in \decset{f_1 \colon A\to X},\,d_2\in \decset{f_2 \colon X\to B} \text{ and }f =_\cat{C} f_1\dcomp f_2
  \end{align*}
\end{definition}
\begin{example}\label{ex:mon-dec}
  Let \(f \colon 1 \to 2\) and \(g \colon 2 \to 1\) be morphisms in a prop such that $\nodeweight(f)=\nodeweight(g)=2$.
  The diagram in \Cref{fig:ex-mon-dec}, left, represents the monoidal decomposition of \(f \dcomp (f \tensor f) \dcomp (g \tensor g) \dcomp g\) given by
  \(\nodegenerator{f}{\dcomp_{2}}{\nodegenerator{\nodegenerator{\nodegenerator{f}{\dcomp_{2}}{g}}{\tensor}{\nodegenerator{f}{\dcomp_{2}}{g}}}{\dcomp_{2}}{g}}\).
\end{example}
\begin{figure}[h!]
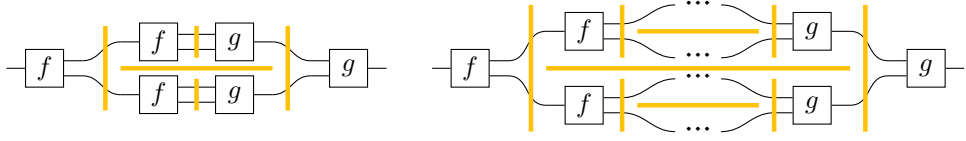

  \centering
  \monoidaldecExFig{}\quad
  \mwdNumberLikeExFig{}
  \caption{Examples of efficient monoidal decompositions.}\label{fig:ex-mon-dec}
\end{figure}

The cost, or \emph{width}, of a decomposition is the weight of the most expensive operation. The decomposition in~\Cref{ex:mon-dec} has width \(2\) as compositions are along at most \(2\) wires.


\begin{definition}[Width of a monoidal decomposition]\label{defn:decompositionWidth}
  Let \(\nodeweight\) be a weight function for \((\cat{C},\decgenerators)\).
  Let $f$ be in $\cat{C}$ and $d\in \decset{f}$.
  The width of $d$ is defined recursively as follows:
  \begin{align*}\label{eq:def-width-mon-dec}
    \decwidth(d) \defn\ &  \nodeweight(f) & \text{if }d=\leafgenerator{f} \\
                        &  \max \{\decwidth(d_1),\decwidth(d_2)\} & \text{if }d= \nodegenerator{d_{1}}{\tensor}{d_{2}}\\
                        &  \max \{\decwidth(d_1),\,\nodeweight(X),\,\decwidth(d_2)\} & \text{if } d= \nodegenerator{d_{1}}{\dcomp_{X}}{d_{2}}
  \end{align*}
\end{definition}

Decompositions can also be described as labelled trees $(S,\mu)$ where $S$ is a tree  and $\mu:\vertices(S)\to \decgenerators \union \monoidaloperations{\cat{C}}$ is a labelling function. The cost can be written
\(\decwidth(d)=\decwidth(S,\mlabelling) \defn \max_{v \in \vertices(S)} \nodeweight(\mlabelling(v))\), which may be familiar to those aquainted with widths.

Monoidal width is simply the width of the cheapest decomposition.
\begin{definition}[Monoidal width]\label{defn:monoidalWidth}
  Let \(\nodeweight\) be a weight function for \((\cat{C},\decgenerators)\)
  and $f$ be in $\cat{C}$. Then the \emph{monoidal width} of $f$ is
  \( \mwd(f) \defn \min_{d\in \decset{f}} \decwidth(d)\).
\end{definition}

\begin{example}\label{ex:mwd-number-like-morphisms}
  With the data of~\Cref{ex:mon-dec}, define an family of morphisms \(h_{n} \colon 1 \to 1\) recursively by \(h_{0} \defn f \dcomp_{2} g\) and \(h_{n+1} \defn f \dcomp_{2} (h_{n} \tensor h_{n}) \dcomp_{2} g\).
  Each \(h_{n}\) has a decomposition of width \(2^{n}\) where the first node is the composition along the \(2^{n}\) wires in the middle.
  However, \(\mwd(h_{n})=2\) for any \(n\), with an optimal monoidal decomposition shown in \Cref{fig:ex-mon-dec}, right.
\end{example}

\subsection{The width of copying}

As a first taste of monoidal width, we study symmetric monoidal categories where some objects $X$ carry a ``copying'' operation, i.e. a morphism \(\cp_{X} \colon X \to X \tensor X\), compatible with the monoidal product.
We show that the copy morphism \(\cp_{X_{1} \tensor \dots \tensor X_{n}}\) on \(X_{1} \tensor \cdots \tensor X_{n}\) can be efficiently decomposed: its monoidal width is bounded by \((n+1) \cdot \max_{i = 1,\dots,n} \nodeweight(X_{i})\).

\begin{definition}[Copying]
  Let \(\cat{X}\) be a symmetric monoidal category with symmetries given by \(\swap{X,Y}\).
  We say that \(\cat{X}\) has \emph{coherent copying} if there is a class of objects $\mathcal{C}_\cat{X}\subseteq \obj{\cat{X}}$,
  satifying $X,Y\in \mathcal{C}_\cat{X}$ iff $X\tensor Y \in \mathcal{C}_{X}$, such that
  every \(X\) in \(\mathcal{C}_\cat{X}\)
   is endowed with a morphism \(\cp_{X} \colon X \to X \tensor X\).
   Moreover, \(\cp_{X \tensor Y} = (\cp_X \tensor \cp_Y) \dcomp (\id{X} \tensor \swap{X,Y} \tensor \id{Y})\) for every \(X,Y \in \mathcal{C}_\cat{X}\).
\end{definition}

An example is any cartesian prop\footnote{In a \emph{cartesian prop} the $\tensor$ satisfies the universal property of  products.
} with the canonical diagonal \(\cp_{X} \colon X \to X \times X\) given by the cartesian structure.
We take \(\cp_{X}\), the symmetries \(\swap{X,Y}\) and the identities \(\id{X}\) as atomic, i.e.\ the set of atomic morphisms is \(\decgenerators = \{\cp_{X},\, \swap{X,Y},\, \id{X} \ :\  X, Y \in \mathcal{C}_\cat{X}\}\).
The weight function is \(\nodeweight(\cp_{X}) \defn 2 \cdot \nodeweight(X)\), \(\nodeweight(\swap{X,Y}) \defn \nodeweight(X) + \nodeweight(Y)\) and \(\nodeweight(\id{X}) \defn \nodeweight(X)\).
Note that $\nodeweight(\cp_{X\tensor Y}) = 2 \cdot \nodeweight(X\tensor Y) =
2\cdot (\nodeweight(X) + \nodeweight(Y))$, but 
we can do better.
\begin{example}\label{ex:mwd-copy}
  Let \(\cat{C}\) be a prop with coherent copying and consider \(\cp_{n} \colon n \to 2n\).
  Let \(\gamma_{n,m} \defn (\cp_{n} \tensor \id{m}) \dcomp (\id{n} \tensor \swap{n,m}) \colon n + m \to n + m + n\).
  We decompose \(\gamma_{n,m}\) (below left) in terms of \(\gamma_{n-1,m+1}\) (in the dashed box), \(\cp_{1}\) and \(\swap{1,1}\) by cutting along at most \(n+1+m\) wires.
  \begin{figure}[h!]
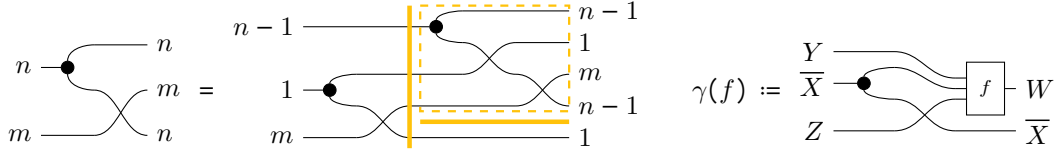

    \centering
    \mwdCopyExFig{} \(\quad \gamma(f) \,\defn\)\lemmamwdcopyStateFig{}
    \caption{Decomposing \(\gamma_{n,m}\) and its more general version \(\gamma(f)\).}\label{fig:mwd-copy}
  \end{figure}
  We decompose \(\cp_{n} = \gamma_{n,0}\) cutting along only \(n+1\) wires. This means that $\mwd(\delta_n) \leq n+1$.
\end{example}
The following result generalises the above and will be useful to prove the results of~\Cref{sec:mwd-branch-width}.
\begin{lemma}\label{lemma:mwd-copy}
  Let \(\cat{X}\) be a sym.\ mon.\ category with coherent copying.
  Suppose that \(\decgenerators\) contains \(\cp_X\) for \(X \in \mathcal{C}_{\cat{X}}\), and \(\swap{X,Y}\) and \(\id{X}\) for \(X \in \obj{\cat{X}}\). Let $\overline{X}\defn X_1 \tensor \cdots \tensor X_n$,
  \(f \colon Y \tensor \overline{X} \tensor Z \to W\), and \(d \in \decset{f}\).
  Let \(\gamma(f) \defn (\id{Y} \tensor \cp_{\overline{X}} \tensor \id{Z}) \dcomp (\id{Y \tensor \overline{X}} \tensor \swap{\overline{X}, Z}) \dcomp (f \tensor \id{\overline{X}})\), as in \Cref{fig:mwd-copy}.
  There is \(\copyMdec(d)\in \decset{g}\) s.t.\ \(\decwidth(\copyMdec(d)) \leq \max \{\decwidth(d), \nodeweight(Y) + \nodeweight(Z) + (n+1) \cdot \max_{i = 1,\ldots,n} \nodeweight(X_i)\}\).
\end{lemma}
\begin{proof}[Proof sketch]
  The proof is by induction on \(n\), the details can be found in~\Cref{app:copy}.
  The intuition for inductive step is the decomposition of~\Cref{ex:mwd-copy}.

\end{proof}
\subsection{Monoidal tree and path widths}
\begin{wrapfigure}{r}{0.25\textwidth}
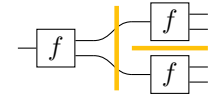

  \centering
  \monoidaltreedecExFig{}
  \caption{A monoidal tree decomposition.}\label{fig:mon-tree-dec}
\end{wrapfigure}
Monoidal width (\Cref{defn:monoidalWidth}) relies on the notion of decomposition (\Cref{defn:monoidalDecomposition}).
We shall consider two variants, obtained by restricting the set of allowable decompositions. First,
a monoidal \emph{tree} decomposition is a monoidal decomposition that has a ``tree'' shape.
Such monoidal decompositions arise by restricting sequential composition to atomic morphisms on one side, with recursion only allowed on the other.
For example, consider a prop with generator \(f \colon 1 \to 2\) and the decomposition of \(f \dcomp (f \tensor f)\) in \Cref{fig:mon-tree-dec}.
The decompositions in~\Cref{fig:ex-mon-dec} are not of this kind.
\begin{definition}
Let \(\cat{C}\) be a monoidal category and \(\decgenerators\) be a set of morphisms.
The set $\decset{f}^{rt}$ of \emph{monoidal (right) tree decompositions} of \(f \colon A \to B\) in $\cat{C}$ is defined recursively:
\begin{align*}
\decset{f}^{rt} \quad \Coloneqq \quad & \leafgenerator{f} &\text{if } f \in \decgenerators \\
                    \mid \quad & \nodegenerator{d_{1}}{\tensor}{d_{2}} &\text{if } d_1\in \decset{f_1}^{rt},\,d_2\in \decset{f_2}^{rt} \text{ and } f =_\cat{C} f_1\tensor f_2 \\
                    \mid \quad & \nodegenerator{g}{\dcomp_{C}}{d_1} &\text{if }g\colon A\to C \in \decgenerators,\,d_1\in \decset{f_1 \colon C\to B}^{rt} \text{ and }f =_\cat{C} g \dcomp f_1
\end{align*}
\end{definition}
The set $\decset{f}^{lt}$ of monoidal left tree decompositions of $f$ is defined analogously, with recursion allowed on the left of `$\dcomp$'.
Since monoidal tree decompositions are examples of monoidal decompositions, we can calculate their width as in~\cref{defn:decompositionWidth}.

\begin{definition}[Monoidal tree width]\label{defn:monoidalTreeWidth}
  Let \(\nodeweight\) be a weight function for \((\cat{C},\decgenerators)\)
  and $f$ be in $\cat{C}$. Then the \emph{monoidal right tree width} and the  \emph{monoidal left tree width} of $f$ are
  \begin{align*}
    \mtwd^{r}(f) &\defn \min_{d\in \decset{f}^{rt}} \decwidth(d) &\text{and} && \mtwd^{l}(f) \defn \min_{d\in \decset{f}^{lt}} \decwidth(d).
  \end{align*}
\end{definition}
In certain categories, e.g. compact closed categories, we can translate between left and right monoidal tree decompositions, preserving width; in such settings we can drop the ``left'' and ``right'' adjectives and talk simply of monoidal tree width. This is the case for cospan categories, and we shall see that in the categories of cospans of graphs, monoidal tree width is closely related to tree width in the traditional sense.

\medskip
\begin{wrapfigure}{r}{0.35\textwidth}
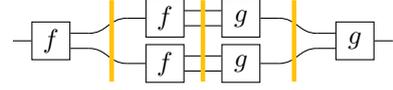

  \setlength{\belowcaptionskip}{-10pt}
  \centering
  \monoidalpathdecExFig{}
  \caption{A monoidal path decomposition.}\label{fig:mon-path-dec}
\end{wrapfigure}
A monoidal path decomposition is a monoidal decomposition that has a ``path'' shape. Conceptually it is very simple: while monoidal tree decompositions restrict the use of $\otimes$, monoidal path decompositions ban it outright.
For example, consider a prop with generators \(f \colon 1 \to 2\), \(g \colon 2 \to 1\). \Cref{fig:mon-path-dec} shows a monoidal path decomposition of \(f \dcomp (f \tensor f) \dcomp (g \tensor g) \dcomp g\).
Here, tensors have to be in the set $\decgenerators$ of atomic morphisms.

\begin{definition}
Let \(\cat{C}\) be a monoidal category and \(\decgenerators\) be a set of morphisms.
The set $\decset{f}^{p}$ of \emph{monoidal path decompositions} of \(f \colon A \to B\) in $\cat{C}$ is defined recursively:
\begin{align*}
\decset{f}^{p} \quad \Coloneqq \quad & \leafgenerator{f} &\text{if } f \in \decgenerators \\
              \mid \quad & \nodegenerator{d_{1}}{\dcomp_{C}}{d_{2}} &\text{if }d_1\in \decset{f_1 \colon A\to C}^p,\,d_2\in \decset{f_2 \colon C\to B}^p \text{ and }f =_\cat{C} f_1\dcomp f_2
\end{align*}
\end{definition}
Since, as for monoidal tree width, we are restricting the set of allowed decompositions, the width of a monoidal path decomposition is inherited from~\cref{defn:decompositionWidth}. This leads, by the now familiar pattern, to monoidal path width.

\section{Graphs and their decompositions}\label{sec:graphs}
We recall the notions of \emph{tree}~\cite{robertson1986graph-minorsII}, \emph{path}~\cite{robertson1983graph-minorsI} and \emph{branch width}~\cite{robertson1991graph-minorsX} due to Robertson and Seymour.
We recall the original definitions, and provide a recursive account of decompositions. The latter allow us, in~\Cref{sec:mwd-graph-widths},
 to establish connections with monoidal width and its variants.

\subsection{Graphs and graphs with sources}

Robertson and Seymour work with finite undirected graphs.
\begin{definition}[Graphs]
  A graph is a triple $G=\mathgraph[,\edgeendsfun]{E}{V}$ where \(V, E\) are finite sets and $\edgeendsfun \colon E \to \parti_2(V)$ sends edges to their endpoints; $\parti_2(V)$ denotes subsets of cardinality 1 or 2.
\end{definition}
We usually write simply \(G = \mathgraph{E}{V}\). 
Fixing $G$, a \emph{path} from $v_0$ to $v_k$ is a list of distinct edges \(e_{0},\ldots,e_{k-1}\) s.t.\  \(\edgeends{e_{i}} = \{v_{i},v_{i+1}\}\) for each \(0 \leq i < k\). 
We write $v \graphpath v'$ if there exists a path from $v$ to $v'$. A graph is \emph{acyclic} if, for each path $v_0\dots v_k$, if $v_i=v_j$ then $i=j$.
A graph is \emph{connected} if for all $v\neq v'\in V$ we have $v \graphpath v'$.
A \emph{tree} is a connected acyclic graph.

For recursive definitions we follow Courcelle~\cite{bauderon1987graph} and recall graphs with sources.
Sources are marked vertices, thought of as an \emph{interface} that can be glued with that of another graph.

\begin{definition}[Graph with sources]
A graph with sources is a pair \(\Gamma = (G,X)\) where \(G = \mathgraph{E}{V}\) is a graph and \(X \subseteq V\) are the sources.
Given \(\Gamma = (G,X)\), \(\Gamma' = (G',X')\), we say that \(\Gamma'\) is a subgraph of \(\Gamma\) if \(G'\) is a subgraph of \(G\) (note that we \emph{do not} require \(X' \subseteq X\)).
\end{definition}

\subsection{Tree width}

Intuitively, \emph{tree width} measures \virgolette{how far} $G$ is from being a tree: a graph is a tree iff it has tree width 1. On the other hand, the family of cliques has unbounded tree width.

Tree width relies on \emph{tree decompositions}. For Robertson and Seymour~\cite{robertson1986graph-minorsII},
a decomposition is itself a tree $Y$, each vertex of which is associated with a subgraph of $G$.
%

\begin{definition}[Tree decomposition]\label{def:treeDec}
  Let \(G = \mathgraph{E}{V}\) a graph.
  A tree decomposition of \(G\) is a pair \((Y,t)\) where \(Y\) is a tree and \(t \colon \vertices(Y) \to \parti(V)\) is a function such that:
  \begin{enumerate}
    \item Every vertex is in one of the components:
      \(\Union_{i \in \vertices(Y)} t(i) = V\).
    \item Every edge has its endpoints in a component:
      \(\forall e \in E \ \exists i \in \vertices(Y) \ \edgeends{e} \subseteq t(i)\).
    \item The components are glued in a tree shape:
      \(\forall i \graphpath j \graphpath k \in \vertices(Y) \ t(i) \intersection t(k) \subseteq t(j)\).
  \end{enumerate}
\end{definition}

\Cref{fig:tree-dec} is an illustration.
The components of the tree are \virgolette{glued along vertices} to obtain the graph.
The cost is the size of the biggest component; e.g.\ in~\Cref{fig:tree-dec} it is \(3\).

\begin{definition}[Tree width]\label{defn:treewidth}
  The width of a tree decomposition is \(\decwidth(Y,t) \defn \max_{i \in \vertices(Y)} \card{t(i)}\).
  The \emph{tree width} of \(G\) is given by the min-max formula:
  \(\treewidth(G) \defn \min_{(Y,t)} \decwidth(Y,t)\).
\end{definition}
Robertson and Seymour subtract \(1\) from $\treewidth(G)$ so that trees have tree width \(1\). To reduce bureaucratic overhead, we ignore this and work with~\Cref{defn:treewidth}.
Robertson and Seymour's decomposition trees are of a graph theoretic kind.
For us it is convenient to use an approach inspired by \(k\)-sourced graphs~\cite{arnborg1993algebraic}.
We consider decompositions as elements of a tree data type, with vertices carrying nonempty subsets $W$ of the vertices $V$ of $G$:
\(T_V \ \Coloneqq\ \emptydec  \ \mid\  (T_V, W, T_V)\).
The function $\labelling$ extracts the relevant subset from the root: $\labelling()\defn\varnothing$, $\labelling(T_1,W,T_2)\defn W$.
\begin{definition}[Recursive tree decomposition]\label{def:recursiveTreeDec}
  Fix a graph with sources \(\Gamma = (\mathgraph{E}{V},X)\).
  A tree decomposition is
  $T\in T_V$ where either $\Gamma$ is empty and $T=()$ or
  \(T = \nodegenerator{T_1}{V'}{T_2}\) and \(T_i\) are decompositions of subgraphs \(\Gamma_i = (\mathgraph{E_i}{V_i},X_i)\) of \(\Gamma\) s.t.:
    (i) \(X \subseteq V'\),
    (ii) \(V' \union V_1 \union V_2 = V\),
    (iii) \(X_i = V_i \intersection V'\),
    (iv) \(V_1 \intersection V_2 \subseteq V'\),
    (v) \(E_{1} \intersection E_{2} = \emptyset\), and
    (vi) \(\edgesetends{E \setminus (E_1 \disjointunion E_2)} \subseteq V'\).
\end{definition}
The conditions ensure that, by glueing \(\Gamma_{1}\) and \(\Gamma_{2}\) with \(\Gamma' \defn (\mathgraph{E\setminus(E_{1} \disjointunion E_{2})}{V'}, X_{1} \union X_{2})\), we get \(\Gamma\) back.
The width is defined:
  \(\decwidth() \defn 0\), and \(\decwidth(T_1,V',T_2) \defn \max\{\card{V'},\decwidth(T_1), \decwidth(T_2)\}\).

\medskip
Tree decompositions of~\Cref{defn:treewidth} and
recursive decompositions of~\Cref{def:recursiveTreeDec} can be translated one to the other while preserving width.
%

\begin{lemma}[cf.\ Proposition~4.1 in~\cite{arnborg1993algebraic}]\label{lemma:equivRecursiveTreeDec}
  Let 
  \(\Gamma = (G,X)\) be a graph with sources.
  \begin{enumerate}
  \item
  If \((Y,t)\) is a tree decomposition of \(G\) and \(X \subseteq t(r)\) for some \(r \in \vertices(Y)\), then there is a recursive tree decomposition \(\toRecursiveDec(Y,t)\) of \(\Gamma\) such that \(\decwidth(Y,t) = \decwidth(\toRecursiveDec(Y,t))\).
  \item
  If \(T\) is a recursive tree decomposition of \(\Gamma\), then there is a tree decomposition \(\fromRecursiveDec(T)\) of \(G\) such that \(X \subseteq \labelling(T)\) and \(\decwidth(T) = \decwidth(\fromRecursiveDec(T))\).
  \end{enumerate}
\end{lemma}
\begin{proof}
  See~\Cref{app:graphs}.
\end{proof}

\subsection{Path width}
Similarly to the intuition behind tree width, path width measures \virgolette{how far} a graph is from being a path.
Graphs are decomposed into a list of subgraphs that form a path shape.

\begin{definition}[Path decomposition]\label{def:pathDec}
  Let \(G = \mathgraph{E}{V}\) be a graph.
  A path decomposition of \(G\) is given by a pair \((P,p)\) where \(P\) is a path and \(p \colon \vertices(P) \to \parti(V)\) is s.t.:
  \begin{enumerate}
    \item Every vertex is in one of the components:
      \(\Union_{i \in \vertices(P)} p(i) = V\).
    \item Every edge has its endpoints in a component:
      \(\forall e \in E \ \exists i \in \vertices(P) \ \edgeends{e} \subseteq p(i)\).
    \item The components are glued in a path shape:
      \(\forall i \graphpath j \graphpath k \in \vertices(P) \ p(i) \intersection p(k) \subseteq p(j)\).
  \end{enumerate}
\end{definition}
\begin{definition}[Path width]
  Given a path decomposition \((P,p)\) of \(G\), \(\decwidth(P,p) \defn \max_{i \in \vertices(P)} \card{p(i)}\).
  The \emph{path width} of \(G\) is given by the min-max formula:
  \(\pathwidth(G) \defn \min_{(P,p)} \decwidth(P,p)\).
\end{definition}

As for tree width, Robertson and Seymour subtract \(1\) from this number to get that paths have path width \(1\). Again, we ignore this convention to save on bureaucracy.

\begin{example}\label{ex:path-dec}
 Components are \virgolette{glued along vertices} (in a path shape) to obtain the original graph. The cost of the decomposition below is \(3\),  the size of the biggest component.
  \begin{center}
    \pathDecExFig{}
  \end{center}
\end{example}

To give a recursive description,
it is again useful to consider a data type of vertex set-labelled paths; in the following $W$ ranges over nonempty subsets of $V$:
\(P_V \ \Coloneqq\ \emptydec  \ \mid\  (W, P_V)\).

\begin{definition}[Recursive path decomposition]\label{def:recursivePathDec}
  Let \(\Gamma = (\mathgraph{E}{V},X)\) be a graph with sources.
  A path decomposition $P\in P_V$ of \(\Gamma\) is either $()$ if $\Gamma$ is empty or
  $(V_1,T')$ where $T'$ is a path decomposition of a subgraph \(\Gamma' = (\mathgraph{E'}{V'},X')\) of \(\Gamma\) and \(V_1 \subseteq V\) is a subset of the vertices of \(\Gamma\), such that:
    (i) \(X \subseteq V_1\),
    (ii) \(V_1 \union V' = V\),
    (iii) \(X' = V_1 \intersection V'\), and
    (iv) \(\edgesetends{E \setminus E'} \subseteq V_1\).
\end{definition}
These conditions ensure that, by glueing \(\Gamma'\) with \(\Gamma_1 \defn (\mathgraph{E \setminus E'}{V_{1}}, X \union X')\), we obtain \(\Gamma\).
  The width can be defined recursively: \(\decwidth() \defn 0\), and \(\decwidth(V_1,T') \defn \max\{\card{V_1},\decwidth(T')\}\).
We state a path version of \Cref{lemma:equivRecursiveTreeDec}: the two notions of decomposition can be used interchangeably.
\begin{lemma}\label{lemma:equivRecursivePathDec}
  Let 
  \(\Gamma = (G,X)\) be a graph with sources.
 \begin{enumerate}
 \item If \((P,p) = (V_1,\ldots,V_r)\) is a path decomposition of \(G\) and \(X \subseteq V_1\), then there is a recursive path decomposition \(\toRecursiveDec(P,p)\) of \(\Gamma\) such that \(\decwidth(P,p) = \decwidth(\toRecursiveDec(P,p))\).
 \item  If $T$ is a recursive path decomposition of \(\Gamma\), then there is a path decomposition \(\fromRecursiveDec(T)\) of \(G\) such that $X \subseteq \labelling(T)$ and $\decwidth(T) = \decwidth(\fromRecursiveDec(T))$.
 \end{enumerate}
\end{lemma}
\begin{proof}
  See~\Cref{app:graphs}.
\end{proof}

\subsection{Branch width}

While related to tree width\footnote{In fact \(\max\{\branchwidth(G),2\} \leq \treewidth(G) \leq \max\{\frac{3}{2} \branchwidth(G),2\}\)~\cite{robertson1991graph-minorsX}.}, branch width is different in spirit to tree and path widths.
A branch decomposition splits graphs into one-edge subgraphs. Given a graph $\mathgraph{E}{V}$ and $v\in V$, the \emph{neighbours} set of $v$ is $\{\, w \suchthat \exists e \in E \ \edgeends{e} = \{v,w\}\,\}$. A \emph{subcubic} tree~\cite{robertson1991graph-minorsX} is a tree s.t.\ each vertex has at most three neighbours. Vertices with one neighbour are \emph{leaves}.

\begin{definition}[Branch decomposition]\label{def:branchDec}
  Let \(G = \mathgraph{E}{V}\) be a graph.
  A branch decomposition \((Y,b)\) of \(G\) has \(Y\) a subcubic tree and \(b \colon \leaves(Y) \cong E\) is a bijection. 
\end{definition}

\begin{example}\label{ex:branch-dec}
  If we choose an edge of \(Y\) to be the \virgolette{root} of the tree, we can extend the labelling to internal vertices by labelling them with the glueing of the labels of their children.
  \begin{center}
    \branchDecExFig{}
  \end{center}
\end{example}
Each edge \(e\) in the tree \(Y\) determines a splitting of the graph.
More precisely, it determines a partition of the leaves of \(Y\), which, through \(b\), determines a partition \(\{A_e,B_e\}\) of the edges of \(G\).
This splits the graph \(G\) into subgraphs \(G_{1}\) and \(G_{2}\).
Intuitively, the order of an edge \(e\) is the number of vertices that are glued together when joining \(G_{1}\) and \(G_{2}\) to get \(G\).
Given the partition \(\{A_e,B_e\}\) of the edges of \(G\), we say that a vertex \(v\) of \(G\) \emph{separates} \(A_e\) and \(B_e\) whenever there are an edge in \(A_e\) and an edge in \(B_e\) that are both adjacent to \(v\).
In~\Cref{ex:branch-dec}, there is one vertex separating the subgraphs of the partition indicated by the arrow: the edge  marked by the arrow has order \(1\).
\begin{definition}[Order of an edge]
  Let \((Y,b)\) be a branch decomposition and \(e\) be an edge of \(Y\).
  The order of \(e\) is
  \(\edgeorder(e) \defn \card{\edgesetends{A_e} \intersection \edgesetends{B_e}}\).
\end{definition}
The width of a branch decomposition is the maximum order of the edges of the tree.
\begin{definition}[Branch width]
  The width of a branch decomposition \((Y,b)\) is \(\decwidth(Y,b) \defn \max_{e \in \edges(Y)} \edgeorder(e)\).
  The \emph{branch width} of \(G\) is: 
  \(\branchwidth(G) \defn \min_{(Y,b)} \decwidth(Y,b)\).
\end{definition}

The recursive definition of branch decomposition is a binary tree whose vertices carry subgraphs \(\Gamma'\) of the ambient graph \(\Gamma\).
We define this set as follows, where \(\Gamma'\) ranges over the non-empty subgraphs of \(\Gamma\):
\(T_{\Gamma} \ \Coloneqq \ \emptydec \ \mid \ (T_{\Gamma}, \Gamma', T_{\Gamma})\).

\begin{definition}[Recursive branch decomposition]
  Let \(\Gamma = (\mathgraph{E}{V},X)\) be a graph with sources.
  A recursive branch decomposition of \(\Gamma\) is \(T \in T_{\Gamma}\) where either:
  \(\Gamma\) is discrete, i.e. it has no edges, and \(T = \emptydec\);
  or \(\Gamma\) has one edge and \(T = \leafgenerator{\Gamma}\);
  or \(T = \nodegenerator{T_1}{\Gamma}{T_2}\) and \(T_i \in T_{\Gamma_{i}}\) are recursive branch decompositions of subgraphs \(\Gamma_i = (\mathgraph{E_i}{V_i},X_i)\) of \(\Gamma\) such that:
    (i) \(E = E_1 \disjointunion E_2\),
    (ii) \(V = V_1 \union V_{2}\), and
    (iii) \(X_i = (V_1 \intersection V_2) \union (X \intersection V_i)\).
\end{definition}
The conditions on the subgraphs ensure that, by glueing \(\Gamma_1\) and \(\Gamma_2\) together, we get \(\Gamma\) back.
Note that \(\edgeends{E_i} \subseteq V_i\).
We will sometimes write \(\Gamma_i = \labelling(T_i)\), \(V_i = \vertices(\Gamma_i)\) and \(X_i = \boundary(\Gamma_i)\).
Then, \(\boundary(\Gamma_i) = (\vertices(\Gamma_1) \intersection \vertices(\Gamma_2)) \union (\boundary(\Gamma) \intersection \vertices(\Gamma_i)) \).

\begin{definition}
  Let \(T = \nodegenerator{T_1}{\Gamma}{T_2}\) be a recursive branch decomposition of \(\Gamma = (G,X)\), with \(T_i\) possibly both empty.
  Define the width of \(T\) recursively: \(\decwidth\emptydec \defn 0\), and \(\decwidth(T) \defn \max\{\decwidth(T_1), \decwidth(T_2),\card{\boundary(\Gamma)}\}\).
  Expanding this, we obtain
  \(\decwidth(T) = \max_{T' \subtreeq T} \card{\boundary(\labelling(T'))}\).
\end{definition}
Also in this case, we show that this definition is equivalent to the original one by exhibiting a width preserving mapping from branch decompositions to recursive branch decompositions. 
Showing that these mappings preserve width is more involved because the order of the edges in a decomposition is defined \virgolette{globally}, while, for a recursive decomposition, the width is defined recursively.
Thus, we first need to show that we can compute recursive width globally.
\begin{lemma}\label{lemma:rec-bwd-globally}
  Let \(\Gamma = (G,X)\) be a graph with sources and \(T\) be a recursive branch decomposition of  \(\Gamma\).
  Let \(T_{0}\) be a subtree of \(T\) and let \(T'\ngtrless T_{0}\) denote a subtree \(T'\) of \(T\) such that its intersection with \(T_{0}\) is empty.
  Then
  \(\boundary(\labelling(T_{0})) = \vertices(\labelling(T_{0})) \intersection (X \union \Union_{T' \ngtrless T_{0}} \vertices(\labelling(T'))).\)
\end{lemma}
\begin{proof}
  See~\Cref{app:graphs}.
\end{proof}
We define two mappings between branch decompositions and recursive branch decompositions.

\begin{lemma}\label{lemma:rec-branch-width-lower}
  Let 
  \(\Gamma = (G,X)\) be a graph with sources.
  Let \(T\) be a recursive branch decomposition of \(\Gamma\).
  There is a branch decomposition \(\fromRecursiveDec(T)\) of \(G\) s.t.\  \(\decwidth(\fromRecursiveDec(T)) \leq \decwidth(T)\).
\end{lemma}
\begin{proof}
  See~\Cref{app:graphs}.
\end{proof}

\begin{lemma}\label{lemma:rec-branch-width-upper}
  Let 
  \(\Gamma = (G,X)\) be a graph with sources.
  Let \((Y,b)\) be a branch decomposition of \(G\).
  Then, there is a branch decomposition \(\toRecursiveDec(Y,b)\) of \(\Gamma\) such that \(\decwidth(\toRecursiveDec(Y,b)) \leq \decwidth(Y,b) + \card{X}\).
\end{lemma}
\begin{proof}
  See~\Cref{app:graphs}.
\end{proof}

\section{Monoidal widths meet graph widths}\label{sec:mwd-graph-widths}
This section contains our main results.
We instantiate monoidal widths in a suitable monoidal category, where arrows are graphs with sources and categorical composition is glueing of graphs. We show that these are, up to constant factors, the graph widths recalled in~\Cref{sec:graphs}.

\subsection{The category \(\cospanGraphO\)}\label{sec:cospan-graph}

We work with the category \(\UGraph\) of undirected graphs and their homomorphisms (\Cref{def:category-graphs}).
The monoidal category \(\cospanGraph\) of cospans is a standard choice for an algebra of \virgolette{open} graphs. Graphs are composed by glueing vertices~\cite{rosebrugh2005cospangraph,gadducci1997inductive,fong2015cospans}.
We do not need the full expressivity of \(\cospanGraph\) and restrict to sets (discrete graphs) as objects.

\begin{definition}
  The category \(\cospanGraphO\) is the full subcategory of \(\cospanGraph\) on discrete graphs.
  Objects are sets and a morphism \(g \colon X \to Y\) is given by a graph \(G = \mathgraph{E}{V}\) and two functions, \(\partial_{X} \colon X \to V\) and \(\partial_{Y} \colon Y \to V\).
\end{definition}

The category \(\cospanGraphO\) has a convenient syntax given by a Frobenius monoid together with an \virgolette{edge} generator \(\oneedge \colon \oneelementset \to \oneelementset\)~\cite{2021feedbackspangraph}.
The Frobenius structure on the objects allows us to copy sources, discard or swap them. The monoid structure arises by taking \(\coproductmap{\id{}}{\id{}}\) and \(\initmap{}\) as the left component of the cospan.
The edge generator is \(\oneedge \colon \oneelementset \to \oneelementset\).

  \begin{align*}
    \cp_{X} & \defn \cospan{X}{\id{}}{\mathgraph{\emptyset}{X}}{\coproductmap{\id{}}{\id{}}}{X \disjointunion X} && \comonoidFig & \delete_{X} & \defn \cospan{X}{\id{}}{\mathgraph{\emptyset}{X}}{\initmap{}}{\emptyset} && \counitFig\\
    \swap{X,Y} & \defn \cospan{X \disjointunion Y}{\swap{}}{\mathgraph{\emptyset}{Y \disjointunion X}}{\id{}}{Y \disjointunion X} && \swapFig & \oneedge & \defn \cospan{\oneelementset}{u}{ \mathgraph{e}{\{u,v\}}}{v}{\oneelementset} && \edgegeneratorFig
  \end{align*}
  where \(\edgeends{e} = \{u,v\}\).
Composition in \(\cospanGraphO\) is given by identification of the common sources: if two vertices are pointed by a common source, then they are identified. For example, the composition \(\oneedge \dcomp \oneedge\) is the path of length two, obtained by identifying the vertex \(v\) of the first morphism with the vertex \(u\) of the second.
%
%
In order to instantiate monoidal widths in \(\cospanGraphO\), we need to define an appropriate weight function.
\begin{definition}
  Let \(\decgenerators\) be all morphisms of \(\cospanGraphO\).
  Define the \emph{weight function} as follows.
  For an object \(X\), \(\nodeweight(\dcompnode{X}) \defn \card{X}\).
  For a morphism \(g \in \decgenerators\), \(\nodeweight(g) \defn \card{V}\), where \(V\) is the set of vertices of the apex of \(g\), i.e. \(g = \cospan{X}{}{G}{}{Y}\) and \(G = \mathgraph{E}{V}\).
\end{definition}
When composing two cospans of graphs, \(g_{1}\) and \(g_{2}\), the apices of \(g_{1}\) and \(g_{2}\) are not necessarily subgraphs of the apex of \(g_{1} \dcomp g_{2}\).
However, we have the following.
\begin{lemma}\label{lemma:episFromComposition}
  Suppose \(g = g_{1} \dcomp g_{2}\) in \(\cospanGraphO\), with \(g = \cospan{X}{}{G}{}{Z}\), \(g_{1}= \cospan{X}{}{G_{1}}{\partial_{1}}{Y}\), \(g_{2} = \cospan{Y}{\partial_{2}}{G_{2}}{}{Z}\).
  Then there are subgraphs \(G'_{1}\) and \(G'_{2}\) of \(G\) and surjections \(\alpha_{i} \colon G_{i} \to G'_{i}\) s.t., if \(\alpha_{iV}(v) = \alpha_{iV}(w)\), then \(v,w \in \image(\partial_{i})\).
\end{lemma}
\subsection{Tree width}\label{sec:mwd-tree-width}
Here we show that monoidal tree width is bounded above by twice tree width and bounded below by tree width, while monoidal path with agrees with path width.
For the tree width case, we establish these bounds by defining two maps: the first, \(\tTomdec\), associates a monoidal tree decomposition to every recursive tree decomposition; the second,  \(\mTotdec\), associates a recursive tree decomposition to every monoidal tree decomposition.

The monoidal tree decomposition \(\tTomdec(T)\) is defined from the recursive tree decomposition \(T\) with its leaves corresponding to the subtrees of \(T\): if \(T' = (T_{1},V',T_{2})\) is a subtree of \(T\) and \(T_{p}\) is its parent tree, then there is a leaf of \(\tTomdec(T)\) whose label is \(g' = \cospan{X'}{}{\mathgraph{E'}{V'}}{}{Y'}\), where \(X' = V' \intersection \labelling(T_{p})\) are the vertices that \(g'\) shares with its parent and \(Y' = V' \intersection (\labelling(T_{1}) \union \labelling(T_{2}))\) are the vertices that \(g'\) shares with its children.
The structure of \(\tTomdec(T)\) is given by the shape of \(T\), as exemplified in the following.
\begin{example}\label{ex:tree-dec-mappings}
  Let \(\Gamma\) be the graph with sources below left.
  The monoidal tree decomposition on the right corresponds to the recursive tree decomposition on the left, where the cospans of graphs \(g_{i}\) are defined as explained above.
  \begin{center}
    \treeDecSourceExFig{} \qquad \qquad
    \mappingTreeDecExFig{}
  \end{center}
\end{example}

The following result formalises this idea.

\begin{proposition}\label{prop:treeDecToCatTreeDec}
  Let \(\Gamma = (G,X)\) be a graph with sources and \(T\) be a tree decomposition of \(\Gamma\).
  Let \(g \defn \cospan{X}{\inclusion}{G}{}{\emptyset}\) be the cospan of graphs corresponding to \(\Gamma\).
  Then, there is \(\tTomdec(T) \in \decset{g}^{rt}\) such that \(\decwidth(\tTomdec(T)) \leq 2 \cdot \decwidth(T)\).
\end{proposition}
\begin{proof}[Proof sketch]
  The proof is by induction on \(T\).
  The base case is easy.
  For the inductive step, let \(T = \nodegenerator{T_{1}}{W}{T_{2}}\).
  Then, \(T_{i}\) are recursive tree decompositions of \(\Gamma_{i} = (G_{i}, X_{i})\) with \(G_{i} = \mathgraph{E_{i}}{V_{i}}\).
  The conditions in the definition of recursive tree decomposition (\Cref{def:recursiveTreeDec}) ensure that \(g\) can be decomposed as \(g = h \dcomp_{X_1 \union X_2} b \dcomp_{X_1 + X_2} (g_1 \tensor g_2)\), where \(h \defn \cospan{X}{}{H}{}{X_{1} \union X_{2}}\), \(H = \mathgraph{E \setminus (E_{1} \disjointunion E_{2})}{W}\), \(g_{i} \defn \cospan{X_{i}}{}{G_{i}}{}{\emptyset}\) and \(b\) is in the dashed box below.
  The maximum cut performed in this decomposition is \(\card{X_{1}} + \card{X_{2}}\), thus, by induction, we can compute the width of this decomposition:\\
  \begin{minipage}{0.5\textwidth}
    \begin{align*}
      & \decwidth(\tTomdec(T)) \\
      & \defn \max \{\card{W},\card{X_1 \union X_2}, \card{X_1} + \card{X_2}, \\
      & \qquad \decwidth(\tTomdec(T_1)), \decwidth(\tTomdec(T_2))\} \\
      & \leq \max \{2 \cdot \card{W}, 2 \cdot \decwidth(T_1),2 \cdot \decwidth(T_2)\} \\
      & = 2 \cdot \decwidth(T)
    \end{align*}
  \end{minipage}
  \begin{minipage}{0.49\textwidth}
    \begin{center}
      \mwdTwdUpperProofFig{}
    \end{center}
  \end{minipage}
  For the details see~\Cref{app:mon-tree-width}.
\end{proof}

Mapping a monoidal tree decomposition to a recursive tree decomposition follows essentially the same idea but requires extra care as the composition of two cospans of graphs \(g_{1} = \cospan{X}{\partial}{G_{1}}{\partial_{1}}{Y}\) and \(g_{2} = \cospan{Y}{\partial_{2}}{G_{2}}{}{\emptyset}\) can identify some vertices in \(G_{1}\) or \(G_{2}\).
This means that \(G_{1}\) and \(G_{2}\) might not be subgraphs of the apex \(G\) of \(g_{1} \dcomp g_{2}\), and, hence, that having a recursive tree decomposition of \((G_{2}, \image(\partial_{2}))\) does not imply having a recursive tree decomposition of \((G,\image(\partial))\).
\Cref{lemma:episPreserveTreeDec} shows that we can define a recursive tree decomposition of the subgraph of \(G\) that corresponds to \((G_{2}, \image(\partial_{2}))\) whose width is bounded by that of a recursive tree decomposition of \((G_{2}, \image(\partial_{2}))\).
This is possible because the mapping \(\alpha_{2} \colon G_{2} \to G\) from~\Cref{lemma:episFromComposition} respects the decomposition structure.


\begin{example}\label{ex:epi-preserve-tree-dec}
  Let \(\Gamma\) be the graph with sources and \(T\) be its recursive tree decomposition shown in~\Cref{ex:tree-dec-mappings}.
  A graph epimorphism respects the structure of this decomposition if it does not identify vertices that are in disjoint components, as shown in \Cref{fig:epis-tree-dec} left.
  \begin{figure}[h!]
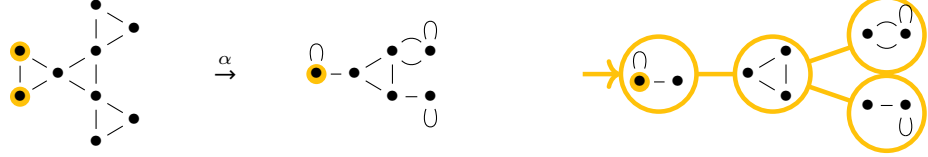

    \centering
    \epiYesExFig{}
    \qquad\qquad\treeDecSourceMapExFig{}
    \caption{Valid epimorphism (left) and corresponding decomposition (right).}\label{fig:epis-tree-dec}
  \end{figure}
  When this condition is met, we can define a recursive tree decomposition of \(\alpha(\Gamma)\) (\Cref{fig:epis-tree-dec}, right) whose width is less than the width of \(T\).
\end{example}

The following lemma constructs a recursive tree decomposition of \(\alpha(\Gamma)\) from a recursive tree decomposition of \(\Gamma\).

\begin{lemma}\label{lemma:episPreserveTreeDec}
  Let \(\Gamma = (G,X)\) and \(\Delta = (H,Y)\) be graphs with sources.
  Let \(T\) be a recursive tree decomposition of \(\Gamma\).
  Let \(\alpha \colon G \to H\) be a graph epimorphism such that, \(\alpha_{V}(X) = Y\) and, if \(\alpha_{V}(v) = \alpha_{V}(w)\), then there is a subtree \(T'\) of \(T\) with \(v,w \in \labelling(T')\).
  Then there is a recursive tree decomposition \(\epiTodec{\alpha}(T)\) of \(\Delta\) such that \(\decwidth(\epiTodec{\alpha}(T)) \leq \decwidth(T)\).
\end{lemma}
\begin{proof}
  See~\Cref{app:mon-tree-width}.
\end{proof}

Thanks to~\Cref{lemma:episPreserveTreeDec}, we can prove the following.

\begin{proposition}\label{prop:catTreeDecToTreeDec}
  Let \(g = \cospan{X}{\partial}{G}{}{\emptyset}\) and \(d \in \decset{g}^{rt}\).
  Let \(\Gamma = (G,\image(\partial))\).
  Then there is a recursive tree decomposition \(\mTotdec(d)\) of \(\Gamma\) such that \(\decwidth(\mTotdec(d)) \leq \max \{\decwidth(d), \card{\image(\partial)}\}\).
\end{proposition}
\begin{proof}[Proof sketch]
  Proceed by induction on \(d\).
  The base case is easily checked.
  If \(d = \nodegenerator{\leafgenerator{h_{1}}}{\dcomp_{Y}}{d_{2}}\), then \(g = h_1 \dcomp_Y h\), with \(h_1 = \cospan{X}{\partial_1}{H_1}{}{Y} \in \decgenerators\) and \(d_2\) monoidal tree decomposition of \(h = \cospan{Y}{\partial_W}{H}{}{\emptyset}\).
  By~\Cref{lemma:episFromComposition}, we can map \(H_{1}\) and \(H\) to subgraphs \(G_{1} = \mathgraph{E_{1}}{V_{1}}\) and \(G_{2} = \mathgraph{E_{2}}{V_{2}}\) of \(G\), respectively.
  By induction there is a recursive tree decomposition \(\mTotdec(d_{2})\) of \((H,\image(\partial_{W}))\).
  By~\Cref{lemma:episPreserveTreeDec}, we can map this decomposition to a decomposition \(\epiTodec{\alpha_{2}}(\mTotdec(d_{2}))\) of \(G_{2}\) without increasing the width.
  Let \(V' \defn \image(\partial) \union (V_{1} \intersection V_{2})\).
  Then, we can define a recursive tree decomposition \(\mTotdec(d) \defn \nodegenerator{V_{1}}{V'}{\epiTodec{\alpha_{2}}(\mTotdec(d_{2}))}\) of \(\Gamma\).
  Its width is:
  \begin{align*}
   \decwidth(\mTotdec(d))
  & \defn \max \{\card{V_1},\card{V'},\decwidth(\epiTodec{\alpha_2}(\mTotdec(d_2)))\}\\
  & \leq \max \{\card{V_1}, \card{Y}, \decwidth(d_2)\}
   \leq \max \{\decwidth(d),\card{\image(\partial)}\}
  \end{align*}
  If, instead, \(d = \nodegenerator{d_1}{\tensor}{d_2}\), then \(d_i\) are monoidal tree decompositions of \(g_i = \cospan{X_i}{\partial_i}{G_i}{}{\emptyset}\) and \(g = g_1 \tensor g_2\).
  We can apply the induction hypothesis to obtain recursive tree decompositions \(\mTotdec(d_{i})\) of \((G_{i}, \image(\partial_{i}))\) and we can define a recursive tree decomposition \(\mTotdec(d) \defn \nodegenerator{\mTotdec(d_1)}{\image(\partial)}{\mTotdec(d_2)}\) of \(\Gamma\).
  Its width can be computed:
  \begin{align*}
   \decwidth(\mTotdec(d))
  & \defn \max \{\card{\image(\partial)}, \decwidth(\mTotdec(d_1)), \decwidth(\mTotdec(d_2))\}
   \leq \max \{\decwidth(d),\card{\image(\partial)}\}
  \end{align*}
  See~\Cref{app:mon-tree-width} for the details of this proof.
\end{proof}

We summarize the results of~\Cref{prop:catTreeDecToTreeDec} and~\Cref{prop:treeDecToCatTreeDec} in the next theorem.

\begin{theorem}\label{th:mwd-tree-width}
  Let \(G\) be a graph and \(g = \cospan{\emptyset}{}{G}{}{\emptyset}\) be the corresponding morphism of \(\cospanGraphO\).
  Then \(\treewidth(G) \leq \mtwd(g) \leq 2 \cdot \treewidth(G)\).
\end{theorem}

\subsection{Path width}\label{sec:path-width}
We can now show that monoidal path width agrees with path width.
%
First, the mapping \(\pTomdec\) from recursive path decompositions to monoidal path decompositions takes the subgraphs of \(\Gamma\) from the recursive decomposition and defines corresponding cospans of graphs.

\begin{proposition}\label{prop:pathDecToCatPathDec}
  Let \((G,X)\) be a graph with sources and \(T\) be a path decomposition.
  There is a monoidal path decomposition \(\pTomdec(T)\) of \(g \defn \cospan{X}{\inclusion}{G}{}{\emptyset}\) s.t.\  \(\decwidth(\pTomdec(T)) = \decwidth(T)\).
\end{proposition}
\begin{proof}
  See~\Cref{app:mon-path-width}.
\end{proof}

The mapping \(\mTopdec\) defines the recursive path decomposition corresponding to a monoidal path decomposition by taking the subgraphs of \(\Gamma\) that correspond to the atoms used in the monoidal decomposition.
We show that these subgraphs respect the path constraints.
Similarly to tree decompositions, we use~\Cref{lemma:episFromComposition} to define the subgraphs of \(\Gamma\) that correspond to the atoms in the monoidal path decomposition.

\begin{lemma}\label{lemma:episPreservePathDec}
  Let \(\Gamma = (G,X)\) and \(\Delta = (H, Y)\) be graphs with sources.
  Let \(T\) be a recursive path decomposition of \(\Gamma\).
  Let \(\alpha \colon G \to H\) be a graph epimorphism such that, \(\alpha_{V}(X) = Y\) and, if \(\alpha_{V}(v) = \alpha_{V}(w)\), then there is a subtree \(T'\) of \(T\) with \(v,w \in \labelling(T')\).
  Then there is a recursive path decomposition \(\epiTodec{\alpha}(T)\) of \(\Delta\) such that \(\decwidth(\epiTodec{\alpha}(T)) \leq \decwidth(T)\).
\end{lemma}
\begin{proof}
  See~\Cref{app:mon-path-width}.
\end{proof}

With this result, we can show the following.

\begin{proposition}\label{prop:catPathDecToPathDec}
  Let \(g = \cospan{X}{\partial}{G}{}{\emptyset}\) and \(d\) be a monoidal path decomposition of \(g\).
  Then there is a recursive path decomposition \(\mTopdec(d)\) of \((G,\image(\partial))\) such that \(\decwidth(\mTopdec(d)) \leq \decwidth(d)\).
\end{proposition}
\begin{proof}
  See~\Cref{app:mon-path-width}.
\end{proof}

\Cref{prop:catPathDecToPathDec} and~\Cref{prop:pathDecToCatPathDec} combine to:

\begin{theorem}\label{th:mwd-path-width}
  Let \(G\) be a graph with cospan \(g = \cospan{\emptyset}{}{G}{}{\emptyset}\).
  %
  Then \(\pathwidth(G) = \mpwd(g)\).
\end{theorem}
\subsection{Branch width}\label{sec:mwd-branch-width}
Here we show that monoidal width is bounded above by branch width \(+ 1\) and bounded below by half of branch width.
We prove these bounds by defining maps from recursive branch decompositions to monoidal decompositions that preserve the width, and vice versa.

The idea behind the mapping from recursive branch decompositions to monoidal decompositions is to take a one-edge graph for each leaf of the recursive branch decomposition and compose them following the structure of the decomposition tree.

\begin{example}
  The \(3\)-clique has a branch decomposition as shown on the left.
  The corresponding monoidal decomposition is shown on the right.
  \begin{center}
    \bwdMappingExFig{}
  \end{center}
\end{example}

The next proposition formalises this procedure, where the \(+ 1\) in the bound comes from the result of~\Cref{lemma:mwd-copy}.
\begin{proposition}\label{prop:mwd-branch-width-upper-bound}
  Let \(\Gamma = (G,X)\) be a graph with sources and \(T\) be a recursive branch decomposition of \(\Gamma\).
  Let \(g \defn \cospan{X}{\inclusion}{G}{}{\emptyset}\) be the corresponding cospan.
  Then, there is \(\bTomdec(T) \in \decset{g}\) such that \(\decwidth(\bTomdec(T)) \leq \decwidth(T) + 1\).
\end{proposition}
\begin{proof}
  See~\Cref{app:mon-branch-width}.
\end{proof}

The mapping \(\mTobdec\) follows the same idea of the mapping \(\bTomdec\) but requires extra care, as in the two previous cases: we need to keep track of which vertices are going to be identified in the final cospan.
The function \(\phi\) stores this information, thus it cannot identify two vertices that are not already in the boundary of the graph.

\begin{proposition}\label{prop:mwd-branch-width-lower-bound}
  Let \(h = \cospan{A}{\partial_A}{H}{\partial_B}{B}\) with \(H = \mathgraph{F}{W}\).
  Let \(\phi \colon W \to V\) such that \(\forall \ w \neq w' \in W \ \phi(w) = \phi(w') \implies w,w' \in \image(\partial_A) \union \image(\partial_B)\) (glueing property).
  Let \(d \in \decset{h}\) and \(\Gamma \defn (\mathgraph{F}{\image(\phi)},\image(\partial_A \dcomp \phi) \union \image(\partial_B \dcomp \phi))\).
  Then, there is a recursive branch decomposition \(\mTobdec(d)\) of \(\Gamma\) such that \(\decwidth(\mTobdec(d)) \leq 2 \cdot \max \{\decwidth(d), \card{A}, \card{B}\}\).
\end{proposition}
\begin{proof}
  See~\Cref{app:mon-branch-width}.
\end{proof}

We summarize these results in the following theorem.
\begin{theorem}
  Let \(G\) be a graph and \(g = \cospan{\emptyset}{}{G}{}{\emptyset}\) be the corresponding morphism of \(\cospanGraphO\).
  Then, \(\frac{1}{2} \cdot \branchwidth(G) \leq \mwd(g) \leq \branchwidth(G) + 1\).
\end{theorem}

\section{Conclusion and future work}
We defined \emph{monoidal width}, a notion of complexity for morphisms in monoidal categories.
Restricting our attention to graphs, we showed that monoidal width and its variants are closely related to existing measures for graphs, namely branch width, tree width and path width. We believe that we can also recover other widths: e.g.\ we conjecture that monoidal path width in the prop of graphs~\cite{chantawibul2015compositionalgraphtheory,networkGamesCSL} is cut width, and monoidal width is rank width.

Future work will also explore other examples of monoidal widths.
For example, we will consider existing generalisations of graph widths in the setting of directed graphs, hypergraphs and relational structures: can they be recovered as instances of monoidal width by suitably changing the monoidal category?


We will also study connections between monoidal width and and algorithmic concerns.
Indeed, inspired by Courcelle's theorem, we are keen to establish general algorithmic results about computing on suitably ``recognisable'' families of arrows of monoidal categories with bounded monoidal width.


\bibliography{ms}
\appendix
\section{Omitted proofs}
\subsection{Copy}\label{app:copy}

\begin{proof}[Proof of~\Cref{lemma:mwd-copy}]
  By induction on \(n\).

  If \(n = 0\), then we can define \(\copyMdec(d) \defn d\).

  Suppose that the statement is true for any \(f' \colon Y \tensor \overline{X} \tensor Z' \to W\).
  Let \(f \colon Y \tensor \overline{X} \tensor X_{n+1} \tensor Z \to W\).
  Then,
  \begin{align*}
    & \gamma_{\overline{X} \tensor X_{n+1}}(f)\\
    \explain{by definition}\\
    & \lemmamwdcopyProofFigOne\\
    \explain{by coherence of \(\cp\)}\\
    & \lemmamwdcopyProofFigTwo\\
    \explain{by naturality of \(\swap{}\) and hexagon equations}\\
    & \lemmamwdcopyProofFigThree
  \end{align*}
  Let \(\gamma_{\overline{X}}(f) \defn (\id{Y} \tensor \cp_{\overline{X}} \tensor \id{X_{n+1} \tensor Z}) \dcomp (\id{Y \tensor \overline{X}} \tensor \swap{\overline{X}, X_{n+1} \tensor Z}) \dcomp (f \tensor \id{\overline{X}})\) be the morphism in the dashed box.
  By the induction hypothesis, there is \(d'\) monoidal decomposition of \(\gamma_{\overline{X}}(f)\) such that \(\decwidth(\copyMdec(d)) \leq \max \{\decwidth(d), \nodeweight(Y) + \nodeweight(X_{n+1} \tensor Z) + (n+1) \cdot \max_{i = 1,\ldots,n} \nodeweight(X_i)\}\).

  Define \(\copyMdec(d)\) as shown below, using the monoidal decomposition \(d'\) of \(\gamma_{\overline{X}}(f)\).
  \begin{center}
    \lemmamwdcopyProofFigCuts{}
  \end{center}
  Then, \(\copyMdec(d)\) is a monoidal decomposition of \(\gamma_{\overline{X} \tensor X_{n+1}}(f)\) because \(\gamma_{\overline{X} \tensor X_{n+1}}(f) = ((\id{Y \tensor \overline{X}}) \tensor ((\cp_{X_n+1} \tensor \id{Z}) \dcomp (\id{X_n+1} \tensor \swap{X_{n+1},Z}))) \dcomp (g' \tensor \id{X_{n+1}})\).
  Moreover,
  \begin{align*}
    & \decwidth(\copyMdec(d)) \\
    \explain{by definition}\\
    & \max \{\nodeweight(\id{Y \tensor \overline{X}}), \nodeweight(\cp_{X_{n+1}}), \nodeweight(\id{Z}), \nodeweight(\id{X_{n+1}}), \nodeweight(\swap{X_{n+1},Z}), \decwidth(d'),\\
    & \qquad \nodeweight(\dcompnode{Y \tensor \overline{X} \tensor Z \tensor X_{n+1}}), \nodeweight(\dcompnode{X_{n+1} \tensor Z \tensor X_{n+1}}) \}\\
    \explain[\leq]{by hypotheses on the weight function}\\
    & \max\{\nodeweight(\dcompnode{Y}) + \nodeweight(\dcompnode{Z}) + \nodeweight(\dcompnode{X_{n+1}}) + \sum_{i=1}^{n+1} \nodeweight(\dcompnode{X_i}), \decwidth(d')\}\\
    & \leq \\
    & \max\{\nodeweight(\dcompnode{Y}) + \nodeweight(\dcompnode{Z}) + (n+2) \cdot \max_{i=1, \ldots, n+1} \nodeweight(\dcompnode{X_i}), \decwidth(d')\} \\
    \explain[\leq]{by induction hypothesis}\\
    & \max\{\nodeweight(\dcompnode{Y}) + \nodeweight(\dcompnode{Z}) + (n+2) \cdot \max_{i=1, \ldots, n+1} \nodeweight(\dcompnode{X_i}), \decwidth(d), \\
    & \qquad \nodeweight(\dcompnode{Y}) + \nodeweight(\dcompnode{X_{n+1} \tensor Z}) + (n+1) \cdot \max_{i=1, \ldots, n} \nodeweight(\dcompnode{X_i})\} \\
    & = \\
    & \max\{\nodeweight(\dcompnode{Y}) + \nodeweight(\dcompnode{Z}) + (n+2) \cdot \max_{i=1, \ldots, n+1} \nodeweight(\dcompnode{X_i}), \decwidth(d)\}
  \end{align*}
\end{proof}

\subsection{Graphs and their decompositions}\label{app:graphs}

\begin{proof}[Proof of~\Cref{lemma:equivRecursiveTreeDec}]
  \((\implies)\) By induction on \(\card{\leaves(Y)}\).

      If \(\card{\leaves(Y)} = 1\), then \((Y,t) = \leafgenerator{V}\) and we can define \(\toRecursiveDec(Y,t) \defn \nodegenerator{\emptydec}{V}{\emptydec}\).

      If \(\card{\leaves(Y)} > 1\), let \(r\) be a vertex of \(Y\) such that \(X \subseteq t(r) \codefn V'\), which is given by hypothesis.
      Let \(v_1,\ldots,v_k\) be the neighbours of \(r\) in \(Y\) and let \(Y_i\) be the subtree of \(Y\) induced by \(v_i\) for \(i = 1,\ldots,k\).
      Let \(V_1 \defn \Union_{n \in \vertices(Y_{1})} t(n)\), \(E_1 \defn \{e \in E \suchthat \edgeends{e} \subseteq V_1\}\) and \(G_1 \defn \mathgraph{E_1}{V_1} \subseteq G\).
      Then, by \Cref{def:treeDec}, \((Y_1,t\restrto{Y_1})\) is a tree decomposition of \(G_1\) and \(X_1 \defn V' \intersection V_1 = V' \intersection t(v_1) \subseteq t(v_1)\).
      We can apply the induction hypothesis as \(\leaves(Y_1) \subsetneq \leaves(Y)\).
      Then, \(T_1 \defn \toRecursiveDec(Y_1,t\restrto{Y_1})\) is a recursive tree decomposition of \(\Gamma_1 \defn (G_1,X_1)\) and \(\decwidth(T_1) = \decwidth(Y_1,t\restrto{Y_1})\).
      Let \(V_2 \defn \Union_{i = 2,\ldots,k} \Union_{n \in \vertices(Y_i)} t(n)\), \(E_2 \defn \{e \in E \setminus E_1 \suchthat \edgeends{e} \subseteq V_2\}\), and \(X_2 \defn V_2 \intersection V'\).
      Let \(G_2 \defn \mathgraph{E_2}{V_2}\) and \((Y',t') \defn \nodegenerator{(Y_2,t\restrto{Y_2})}{X_2}{\cdots(Y_k,t\restrto{Y_k})}\) be the tree obtained by joining \((Y_2,t\restrto{Y_2}), \dots, (Y_k,t\restrto{Y_k})\) in a chain.
      Then, \((Y',t')\) is a tree decomposition of \(G_2\).
      We can apply the induction hypothesis as \(\leaves(Y') \subsetneq \leaves(Y)\).
      Then, \(T_2 \defn \toRecursiveDec(Y',t')\) is a recursive tree decomposition of \(\Gamma_2 \defn (G_2,X_2)\) with \(\decwidth(T_2) = \decwidth(Y',t')\).
      Define \(\toRecursiveDec(Y,t) \defn \nodegenerator{T_1}{V'}{T_2}\).
      Then, \(X \subseteq V'\), \(V' \union V_1 \union V_2 = \Union_{n \in \vertices(Y)} t(n) = V\), \(X_i \defn V_i \intersection V'\), \(V_1 \intersection V_2 \subseteq V'\) (by \Cref{def:treeDec}), \(E_1 \intersection E_2 = \emptyset\) (by construction), and \(\edgesetends{E \setminus (E_1 \disjointunion E_2)} \subseteq V'\) (by \Cref{def:treeDec}).
      Then, \(T\) is a recursive tree decomposition of \(\Gamma\).
      Moreover,
      \begin{align*}
        & \decwidth(\toRecursiveDec(Y,t))\\
        & \defn \max \{\card{V'}, \decwidth(T_1), \decwidth(T_2)\}\\
        & = \max \{\card{V'}, \decwidth(Y_1,t\restrto{Y_1}), \decwidth(Y',t')\} \\
        & \codefn \max \{\card{t(r)}, \max_{n \in \vertices(Y_1)} \card{t(n)}, \max_{n \in \vertices(Y')} \card{t'(n)}\}\\
        & = \max \{\max_{n \in \vertices(Y)}\card{t(n)}, \card{X_2}\}\\
        & = \max_{n \in \vertices(Y)}\card{t(n)}\\
        & \codefn \decwidth(Y,t)
      \end{align*}
    \((\co{\implies})\) By induction on \(T\).

      If \(T = \nodegenerator{\emptydec}{V}{\emptydec}\), then define \(\fromRecursiveDec(T) \defn \leafgenerator{V}\).

      If \(T = \nodegenerator{T_1}{V'}{T_2}\), then \(T_i\) is a recursive tree decomposition of \(\Gamma_i = (G_i,X_i)\) with \(G_i = \mathgraph{E_i}{V_i}\) such that \(X \subseteq V'\), \(V' \union V_1 \union V_2 = V\), \(X_i = V_i \intersection V'\), \(E_1 \intersection E_2 = \emptyset\), and \(\edgesetends{E \setminus (E_1 \disjointunion E_2)} \subseteq V'\).
      Then, \((Y_i,t_i) \defn \fromRecursiveDec(T_i)\) is a tree decomposition of \(G_i\) with \(\decwidth(Y_i,t_i) = \decwidth(T_i)\) by induction hypothesis.
      Let \(\fromRecursiveDec(T) \defn \nodegenerator{(Y_1,t_1)}{V'}{(Y_2,t_2)}\).
      Then, \(\fromRecursiveDec(T)\) satisfies the conditions for being a tree decomposition of \(G\).
      Moreover,
      \begin{align*}
        & \decwidth(\fromRecursiveDec(T)) \\
        & \defn \max_{n \in \vertices(Y)} \card{t(n)}\\
        & = \max \{\card{V'}, \decwidth(Y_1,t_1), \decwidth(Y_2,t_2)\}\\
        & = \max \{\card{V'}, \decwidth(T_1), \decwidth(T_2)\}\\
        & \codefn \decwidth(T)
      \end{align*}
\end{proof}

\begin{remark}
  In the case of monoidal path decompositions, the cost of compositions does not depend on the order in which we perform the compositions.
  Thus, when computing the width of a monoidal path decomposition of a morphism, we can consider the list of morphisms that corresponds to the decomposition.
\end{remark}

\begin{proof}[Proof of~\Cref{lemma:equivRecursivePathDec}]
  \((\implies)\) By induction on \(\pathlength(P)\).

      If \(\pathlength(P) = 1\), then \((P,p) = (V)\) and \(\toRecursiveDec(P,p) = (V)\).

      If \(\pathlength(P) = r\), then \((P,p) = (V_1, \ldots, V_r)\).
      Let \((P',p') \defn (V_2,\ldots,V_r)\), \(V' \defn V_2 \union \ldots \union V_r\), \(X' \defn V_1 \intersection V'\) and \(E' \defn \{e \in E \suchthat \edgeends{e} \subseteq V'\}\).
      Then, \(G' \defn \mathgraph{E'}{V'}\) is a subgraph of \(G\) such that \(X \subseteq V_1\) (by hypothesis), \(V_1 \union V' = \Union_i V_i = V\) (by \Cref{def:pathDec}), \(X' = V_1 \intersection V'\) (by construction), \(\edgesetends{E\setminus E'} \subseteq V_1\) (because \(E \setminus E' \subseteq E_1 \defn \{e \in E \suchthat \edgeends{e} \subseteq V_1\}\) by \Cref{def:pathDec}).
      Then, \((P',p')\) is a path decomposition of \(G'\) and \(X' \defn V_1 \intersection V' = V_1 \intersection V_2 \subseteq V_2\) by \Cref{def:pathDec}.
      Then, \(\toRecursiveDec(P',p')\) is a recursive path decomposition of \(\Gamma' = (G',X')\) with \(\decwidth(\toRecursiveDec(P',p')) = \decwidth(P',p')\) by induction hypothesis.
      Then, \(\toRecursiveDec(P,p) \defn (V_1,\toRecursiveDec(P',p'))\) is a recursive path decomposition of \(\Gamma\).
      Moreover,
      \begin{align*}
        & \decwidth(\toRecursiveDec(P,p)) \\
        & \defn \max \{\card{V_1},\decwidth(\toRecursiveDec(P',p'))\} \\
        & = \max \{\card{V_1},\decwidth(P',p')\} \\
        & \defn \max_{i = 1,\ldots,r} \card{V_i}\\
        & \codefn \decwidth(P,p)
      \end{align*}
    \((\co{\implies})\) By induction on \(T\).

      If \(T = (V)\), then \(\fromRecursiveDec(T) \defn (V)\).

      If \(T = (V_1,T')\) with \(T'\) recursive path decomposition of \(\Gamma' = (\mathgraph{E'}{V'},X')\).
      Then, \(X \subseteq V_1\), \(V_1 \union V' = V\), \(X' = V_1 \intersection V'\) and \(\edgesetends{E\setminus E'} \subseteq V_1\) by \Cref{def:recursivePathDec}.
      By induction hypothesis, \(\fromRecursiveDec(T') = (V_2,\ldots,V_r)\) with \(V' = V_2 \union \ldots \union V_r\), \(\forall e \in E' \ \exists i \in \{2,\ldots,r\} \ \edgeends{e} \subseteq V_i\), \(\forall i,j,k \in \{2,\ldots,r\} \ i \leq j \leq k \implies V_i \intersection V_k \subseteq V_j\), and \(\decwidth(T') = \decwidth(\fromRecursiveDec(T'))\).
      Then, \(V_1 \union \ldots \union V_r = V_1 \union V' = V\), \(\forall e \in E \ \exists i \in \{1,\ldots,r\} \ \edgeends{e} \subseteq V_i\) and \(\forall i,j,k \in \{1,\ldots,r\} \ i \leq j \leq k \implies V_i \intersection V_k \subseteq V_j\).
      Then, \(\fromRecursiveDec(T) \defn (V_1,\ldots,V_r)\) is a path decomposition of \(G\) and \(X \subseteq V_1\).
      Moreover,
      \begin{align*}
        & \decwidth(\fromRecursiveDec(T)) \\
        & \defn \max_{i = 1,\ldots,r} \card{V_i}\\
        & \codefn \max \{\card{V_1}, \decwidth(\fromRecursiveDec(T'))\}\\
        & = \max \{\card{V_1},\decwidth(T')\}\\
        & \codefn \decwidth(T)
      \end{align*}
\end{proof}

\begin{proof}[Proof of~\Cref{lemma:rec-bwd-globally}]
  Proceed by induction on \(T\).

  If \(T = \emptydec\), then \(T_0 = \emptydec\) and we are done.

  If \(T = \nodegenerator{T_1}{\Gamma}{T_2}\), then either \(T_0 \subtreeq T_1\), \(T_0 \subtreeq T_2\) or \(T_0 = T\).
  If \(T_0 = T\), then \(\boundary(\labelling(T_0)) = X\) and we are done.
  Suppose that \(T_0 \subtreeq T_1\).
  Then, by applying the induction hypothesis and using the fact that \(\labelling(T_0) \subgrapheq \labelling(T_1)\), we show
  \begin{align*}
    & \boundary(\labelling(T_0)) \\
    & = \vertices(\labelling(T_0)) \intersection (\boundary(\labelling(T_1))) \union \Union_{T' \subtreeq T_1, T' \ngtrless T_0} \vertices(\labelling(T'))\\
    & = \vertices(\labelling(T_0)) \intersection ((\vertices(\labelling(T_1)) \intersection (\vertices(\labelling(T_2)) \union X)) \union \Union_{T' \subtreeq T_1, T' \ngtrless T_0} \vertices(\labelling(T'))\\
    & = \vertices(\labelling(T_0)) \intersection (\vertices(\labelling(T_2)) \union X \union \Union_{T' \subtreeq T_1, T' \ngtrless T_0} \vertices(\labelling(T'))\\
    & = \vertices(\labelling(T_0)) \intersection (X \union \Union_{T' \subtreeq T, T' \ngtrless T_0} \vertices(\labelling(T'))
  \end{align*}
  We proceed analogously if \(T_0 \subtreeq T_2\).
\end{proof}

\begin{proof}[Proof of~\Cref{lemma:rec-branch-width-lower}]
  A binary tree is, in particular, a subcubic tree.
  Then, we can define \(Y\) to be the unlabelled tree underlying \(T\).
  The label of a leaf \(l\) of \(T\) is a subgraph of \(\Gamma\) with one edge \(e_l\).
  Then, there is a bijection \(b \colon \leaves(T) \to \edges(G)\) such that \(b(l) \defn e_l\).
  Then, \((Y,b)\) is a branch decomposition of \(G\) and we can define \(\fromRecursiveDec(T) \defn (Y,b)\).

  By definition, \(e \in \edges(Y)\) if and only if \(e \in \edges(T)\).
  Let \(\{v,w\} = \edgeends{e}\) with \(v\) parent of \(w\) in \(T\) and let \(T_w\) the subtree of \(T\) with root \(w\).
  Let \(\{E_v,E_w\}\) be the (non-trivial) partition of \(E\) induced by \(e\).
  Then, \(E_w = \edges(\labelling(T_{w}))\) and \(E_v = \Union_{T' \ngtrless T_{w}} \edges(\labelling(T'))\).
  Then, \(\edgesetends{E_w} \subseteq \vertices(\labelling(T_{w}))\) and \(\edgesetends{E_v} \subseteq \Union_{T' \ngtrless T_{w}} \vertices(\labelling(T'))\).
  Using these inclusions and applying \Cref{lemma:rec-bwd-globally},
  \begin{align*}
    & \edgeorder(e) \\
    & \defn \card{\edgesetends{E_w} \intersection \edgesetends{E_v}} \\
    & \leq \card{\vertices(\labelling(T_{w})) \intersection \Union_{T' \ngtrless T_{w}} \vertices(\labelling(T'))} \\
    & \leq \card{\vertices(\labelling(T_{w})) \intersection (X \union \Union_{T' \ngtrless T_{w}} \vertices(\labelling(T')))} \\
    & = \card{\boundary(\labelling(T_{w}))}
  \end{align*}
  Then,
  \begin{align*}
    & \decwidth(Y,b)\\
    & \defn \max_{e \in \edges(Y)} \edgeorder(e)\\
    & \leq \max_{T' < T} \card{\boundary(\labelling(T'))}\\
    & \leq \max_{T' \leq T} \card{\boundary(\labelling(T'))}\\
    & = \decwidth(T)
  \end{align*}
\end{proof}

\begin{proof}[Proof of~\Cref{lemma:rec-branch-width-upper}]
  Proceed by induction on \(\card{\edges(Y)}\).
  
  If \(Y\) has no edges, then either \(G\) has no edges and \((Y,b) = \emptydec\) or \(G\) has only one edge \(e_l\) and \((Y,b) = \leafgenerator{e_l}\).
  In either case, define \(\toRecursiveDec(Y,b) \defn \leafgenerator{\Gamma}\) and \(\decwidth(\toRecursiveDec(Y,b)) \defn \card{X} \leq \decwidth(Y,b) + \card{X}\).

  If \(Y\) has at least one edge \(e\), then \(Y = Y_1 \overset{e}{\text{---}} Y_2\) with \(Y_i\) a subcubic tree.
  Let \(E_{i} = b(\leaves(Y_{i}))\) be the sets of edges of \(G\) indicated by the leaves of \(Y_{i}\).
  Then, \(E_{1} \disjointunion E_{2} = E\).
  By induction hypothesis, there are recursive branch decompositions \(T_{i} \defn \toRecursiveDec(Y_{i},b_{i})\) of \(\Gamma_{i} = (G_{i}, X_{i})\), where \(V_{1} \defn \edgesetends{E_{1}}\), \(V_{2} \defn \edgesetends{E_{2}} \union (V \setminus V_{1})\), \(X_{i} \defn (V_{1} \intersection V_{2}) \union (V_{i} \intersection X)\) and \(G_{i} \defn \mathgraph{E_{i}}{V_{i}}\).
  Moreover, \(\decwidth(T_{i}) \leq \decwidth(Y_{i}, b_{i}) + \card{X_{i}}\).
  Then, the tree \(\toRecursiveDec(Y,b) \defn \nodegenerator{T_1}{\Gamma}{T_2}\) is a recursive branch decomposition of \(\Gamma\) and, by applying~\Cref{lemma:rec-bwd-globally},
  \begin{align*}
    & \decwidth(\toRecursiveDec(Y,b)) \\
    & \defn \max \{\decwidth(T_{1}), \card{X}, \decwidth(T_{2})\}\\
    & = \max_{T' \leq T} \card{\boundary(\labelling(T'))}\\
    & \leq \max_{T' \leq T} \card{\vertices(\labelling(T')) \intersection \edgesetends{E \setminus \edges(\labelling(T'))}} + \card{X} \\
    & = \max_{e \in \edges(Y)} \edgeorder(e) + \card{X}\\
    & \codefn \decwidth(Y,b) + \card{X}
  \end{align*}
\end{proof}

\begin{definition}\label{def:category-graphs}
  The category \(\UGraph\) has (undirected) graphs as objects.
  Given two graphs \(G = \mathgraph{E}{V}\) and \(H = \mathgraph{F}{W}\), a morphism \(\alpha \colon G \to H\) is given by a pair of functions \(\alpha_{V} \colon V \to W\) and \(\alpha_{E} \colon E \to F\) s.t.\ \(\edgesetends[H]{\alpha_{E}(e)} = \alpha_{V}(\edgesetends[G]{e})\).
  \begin{center}
    \begin{tikzcd}
      {E} \arrow[r, "{\alpha_E}"] \arrow[d, "\edgeendsfun_G"'] & {F} \arrow[d, "{\edgeendsfun_H}"] \\
      {\parti_2(V)} \arrow[r, "{\parti_2(\alpha_V)}"]                 & {\parti_2(W)}
    \end{tikzcd}
  \end{center}
\end{definition}

It is easy to show that this category has all colimits, computed pointwise\footnote{Given a pair of functions, \(f \colon A \to B\) and \(g \colon C \to B\), with common codomain, the pushout of \(f\) and \(g\) is given by the set \((B \disjointunion C)/\sim\), which is the disjoint union of \(B\) and \(C\) quotiented by the equivalence relation generated by \(f(a) \sim g(a)\). Pushouts of morphisms of graphs are computed componentwise.}.

\begin{lemma}\label{lemma:colimits-graph}
  The category \(\UGraph\) has all colimits and they are computed pointwise.
\end{lemma}
\begin{proof}
  Let \(D \colon \cat{J} \to \UGraph\) be a diagram in \(\UGraph\).
  Let \(U_{E} \colon \UGraph \to \Set\) and \(U_{V} \colon \UGraph \to \Set\) be the functors that associate the edges, resp. vertices, component to graphs and graph morphisms.
  The category \(\Set\) has all colimits, thus there are \(\colim(D \dcomp U_{E})\) and \(\colim(D \dcomp U_{V})\).
  Let \(c_{i} \colon U_{V}(D(i)) \to \colim(D \dcomp U_{V})\) and \(d_{i} \colon U_{E}(D(i)) \to \colim(D \dcomp U_{E})\) be the inclusions given by the colimits.
  For every \(f \colon i \to j\) in \(\cat{J}\), we have that \(f_{E} \dcomp \edgeendsfun_{j} = \edgeendsfun_{i} \dcomp \parti_{2}(f_{V})\) by definition of graph morphism (\Cref{def:category-graphs}), and \(\parti_{2}(f_{V}) \dcomp \parti_{2}(c_{j}) = \parti_{2}(c_{i})\), by definition of colimit.
  This shows that \(\parti_{2}(\colim(D \dcomp U_{V}))\) is a cocone over \(D \dcomp U_{E}\) with morphisms given by \(\edgeendsfun_{i} \dcomp \parti_{2}(c_{i})\).
  Then, there is a unique morphism \(\edgeendsfun \colon \colim(D \dcomp U_{E}) \to \colim(D \dcomp U_{V})\) that commutes with the cocone morphisms.
  This shows that the pairs \((c_{i},d_{i})\) with the graph defined by \((\colim(D \dcomp U_{E}), \colim(D \dcomp U_{V}), \edgeendsfun)\) form a cocone over \(D\) in \(\UGraph\).
  This cocone is initial because its components are so.
\end{proof}
In particular, coproducts and pushouts exist, and the following is well-known.
\begin{proposition}
  Let \(\cat{C}\) be a category with finite colimits.
  Then there is a symmetric monoidal category \(\catCospan{\cat{C}}\) with the same objects as \(\cat{C}\). Morphisms \(A \to B\) are (isomorphism classes of) \emph{cospans}: pairs $A \xrightarrow{f} C \xleftarrow{g} B$ of $\cat{C}$ morphisms with the same codomain.
  Composition is computed by pushout in \(\cat{C}\) and monoidal product by the coproduct of \(\cat{C}\).
\end{proposition}

\subsection{Monoidal tree width}\label{app:mon-tree-width}

\begin{proof}[Proof of~\Cref{prop:treeDecToCatTreeDec}]
  Proceed by induction on \(T\).

  If the tree has only one node \(T = (V)\), then define \(\tTomdec(T) \defn (g)\) and \(\decwidth(\tTomdec(T)) \defn \card{V} \codefn \decwidth(T)\).

  If \(T = \nodegenerator{T_1}{W}{T_2}\),
  by \Cref{def:recursiveTreeDec}, \(T_i\) is a recursive tree decomposition of \(\Gamma_i = (G_i,X_i)\), where \(G_i = \mathgraph{E_i}{V_i} \subseteq G\), such that \(X \subseteq W \subseteq V\), \(V_1 \intersection V_2 \subseteq W\), \(V_1 \union V_2 \union W = V\), \(X_i = V_i \intersection W\), \(E_1 \intersection E_2 = \emptyset\) and \(\edgesetends{E \setminus (E_1 \disjointunion E_2)} \subseteq W\).
  By induction hypothesis, there are monoidal tree decompositions \(\tTomdec(T_i)\) of \(g_i \defn \cospan{X_i}{\inclusion}{G_i}{}{\emptyset}\) such that \(\decwidth(\tTomdec(T_i)) \leq 2 \cdot \decwidth(T_i)\).
  Let \(F \defn E \setminus (E_1 \disjointunion E_2)\), \(H \defn \mathgraph{F}{W}\) and \(h \defn \cospan{X}{\inclusion}{H}{}{X_1 \union X_2}\).
  Let \(b \defn \id{X_1 \setminus X_2} \tensor \cp_{X_1 \intersection X_2} \tensor \id{X_2 \setminus X_1} \colon X_1 \union X_2 \to X_1 + X_2\).
  Then, \(g = h \dcomp_{X_1 \union X_2} b \dcomp_{X_1 + X_2} (g_1 \tensor g_2)\).
  Then, \(\tTomdec(T) \defn \nodegenerator{h}{\dcomp_{X_1 \union X_2}}{\nodegenerator{b}{\dcomp_{X_1 + X_2}}{\nodegenerator{\tTomdec(T_1)}{\tensor}{\tTomdec(T_2)}}}\) is a monoidal tree decomposition of \(g\):
  \begin{center}
    \mwdTwdUpperProofFig{}
  \end{center}
  We compute the width of \(\tTomdec(T)\).
  \begin{align*}
    & \decwidth(\tTomdec(T)) \\
    & \defn \max \{\nodeweight(h), \nodeweight(\dcompnode{X_1 \union X_2}), \nodeweight(b), \nodeweight(\dcompnode{X_1 + X_2}), \decwidth(\tTomdec(T_1)), \decwidth(\tTomdec(T_2))\}\\
    & \defn \max \{\card{W},\card{X_1 \union X_2}, \card{X_1} + \card{X_2}, \decwidth(\tTomdec(T_1)), \decwidth(\tTomdec(T_2))\} \\
    & \leq \max \{2 \cdot \card{W}, 2 \cdot \decwidth(T_1),2 \cdot \decwidth(T_2)\} \\
    & = 2 \cdot \max \{\card{W}, \decwidth(T_1), \decwidth(T_2)\} \\
    & \codefn 2 \cdot \decwidth(T)
  \end{align*}
\end{proof}

\begin{proof}[Proof of~\Cref{lemma:episPreserveTreeDec}]
  Let \(G = \mathgraph{E}{V}\) and \(H = \mathgraph{F}{W}\).
  Proceed by induction on \(T\).

  If the tree has only one node \(T = (V)\), then, we define \(\epiTodec{\alpha}(T) \defn (W)\) and
  \(\decwidth(\epiTodec{\alpha}(T)) \defn \card{W} \leq \card{V} \codefn \decwidth(T)\) because \(\alpha\) is an epimorphism.

  If the tree is not a leaf, \(T = \nodegenerator{T_1}{V'}{T_2}\), then \(V' \subseteq V\), \(T_i\) is a recursive tree decomposition of \(\Gamma_{i} = (G_{i},X_{i})\), with \(G_i = \mathgraph{E_i}{V_i} \subseteq G\) and \(X_{i} = V_{i} \intersection V'\), such that \(V' \union V_1 \union V_2 = V\), \(V_1 \intersection V_2 \subseteq V'\) and \(\edgesetends{E \setminus (E_1 \union E_2)} \subseteq V'\).
  Let \(G' \defn \mathgraph{E \setminus (E_1 \union E_2)}{V'}\).
  Then, \(\alpha \colon G \to H\) restricts to \(\alpha' \colon G' \to H'\), \(\alpha_1 \colon G_1 \to H_1\) and \(\alpha_2 \colon G_2 \to H_2\) graph epimorphisms for some \(H' = \mathgraph{F'}{W'}\), \(H_1 = \mathgraph{F_1}{W_1}\) and \(H_2 = \mathgraph{F_2}{W_2}\) subgraphs of \(H\). They, moreover, satisfy that, if \(\alpha_{iV}(v) = \alpha_{iV}(w)\), then there is a subtree \(T'\) of \(T_{i}\) with \(v,w \in \labelling_{i}(T')\) because \(\alpha\) satisfies this condition.
  Then \(W' \union W_1 \union W_2 = W\), \(W_1 \intersection W_2 \subseteq W\) and \(\edgesetends{F\setminus (F_1 \union F_2)} \subseteq W'\) because \(\alpha_V\) and \(\alpha_E\) surjective, and, if \(\alpha_{V}(v) = \alpha_{V}(w)\), then there is a subtree \(T'\) of \(T_{i}\) with \(v,w \in \labelling_{i}(T')\).
  By induction there are recursive tree decompositions \(\epiTodec{\alpha_i}(T_i)\) of \((H_i,\alpha_{i}(X_{i}))\) such that \(\decwidth(\epiTodec{\alpha_i}(T_i)) \leq \decwidth(T_i)\).
  Then there is a recursive tree decomposition of \((H,Y)\) given by \(\epiTodec{\alpha}(T) \defn \nodegenerator{\epiTodec{\alpha_1}(T_1)}{W'}{\epiTodec{\alpha_2}(T_2)}\).

  We compute the width of \(\epiTodec{\alpha}(T)\).
  \begin{align*}
  & \decwidth(\epiTodec{\alpha}(T)) \\
  & \defn \max\{\card{W'},\decwidth(\epiTodec{\alpha_1}(T_1)),\decwidth(\epiTodec{\alpha_2}(T_2))\}\\
  & \leq \max \{\card{W'},\decwidth(T_1),\decwidth(T_2)\}\\
  & \leq \max \{\card{V'},\decwidth(T_1),\decwidth(T_2)\}\\
  & \codefn \decwidth(T)
  \end{align*}
\end{proof}

\begin{proof}[Proof of~\Cref{prop:catTreeDecToTreeDec}]
  Let \(G = \mathgraph{E}{V}\). Proceed by induction on \(d\).

  If \(d = \leafgenerator{g}\), then \(\mTotdec(d) \defn \leafgenerator{V}\) and \(\decwidth(\mTotdec(d)) = \card{V} = \decwidth(d) \leq \max \{\decwidth(d), \card{\image(\partial)}\}\).

  If \(d = \nodegenerator{\leafgenerator{h_1}}{\dcomp_Y}{d_2}\), then \(h_1 = \cospan{X}{\partial_1}{H_1}{}{Y} \in \decgenerators\), \(d_2\) is a monoidal tree decomposition of \(h = \cospan{Y}{\partial_W}{H}{}{\emptyset}\) and \(g = h_1 \dcomp_Y h\).
  By induction, there is a recursive tree decomposition \(\mTotdec(d_2)\) of \((H, \image(\partial_{W}))\) with \(\decwidth(\mTotdec(d_2)) \leq \max \{\decwidth(d_2), \card{\image (\partial_W)}\}\).
  By \Cref{lemma:episFromComposition}, there are \(\alpha_2 \colon H \to G_2\) and \(\alpha_1 \colon H_1 \to G_1\) for some \(G_1 = \mathgraph{E_1}{V_1}\) and \(G_2 = \mathgraph{E_2}{V_2}\) subgraphs of \(G\).
  All the vertices of \(H\) that are identified by \(\alpha_{2}\) need to be in \(\image(\partial_{W})\).
  Thus, \(\alpha_{2}\) satisfies the hypothesis of~\Cref{lemma:episPreserveTreeDec}.
  We can, then, apply this lemma to obtain a recursive tree decomposition \(\epiTodec{\alpha_2}(\mTotdec(d_2))\) of \((G_2,\image(\partial_{W} \dcomp \alpha_{2}))\) such that \(\decwidth(\epiTodec{\alpha_2}(\mTotdec(d_2))) \leq \decwidth(\mTotdec(d_2))\).
  Moreover, \(\image(\partial) \subseteq V_1\), \(V' \defn \image(\partial) \union (V_1 \intersection V_2)\), \(X_i \defn V_i \intersection V'\) and \(E = E_1 \disjointunion E_2\).
  Then, \(\mTotdec(d) \defn \nodegenerator{V_1}{V'}{\epiTodec{\alpha_2}(\mTotdec(d_2))}\) is a tree decomposition of \(\Gamma\).

  We compute the width of \(\mTotdec(d)\).
  \begin{align*}
  & \decwidth(\mTotdec(d)) \\
  & \defn \max \{\card{V_1},\card{V'},\decwidth(\epiTodec{\alpha_2}(\mTotdec(d_2)))\}\\
  & = \max \{\card{V_1},\decwidth(\epiTodec{\alpha_2}(\mTotdec(d_2)))\}\\
  & \leq \max \{\card{V_1},\decwidth(\mTotdec(d_2))\}\\
  & \leq \max \{\card{V_1},\decwidth(d_2), \card{\image(\partial_W)}\}\\
  & \leq \max \{\card{V_1}, \card{Y}, \decwidth(d_2)\}\\
  & \codefn \decwidth(d)\\
  & \leq \max \{\decwidth(d),\card{\image(\partial)}\}
  \end{align*}

  If \(d = \nodegenerator{d_1}{\tensor}{d_2}\), then \(d_i\) are monoidal tree decompositions of \(g_i = \cospan{X_i}{\partial_i}{G_i}{}{\emptyset}\) and \(g = g_1 \tensor g_2\).
  Let \(G_i = \mathgraph{E_i}{V_i}\).
  Then, \(V = V_1 \disjointunion V_2\), \(E = E_1 \disjointunion E_2\) and \(G_1, G_2 \subseteq G\).
  By induction, there are recursive tree decompositions \(\mTotdec(d_i)\) of \((G_i,\image(\partial_{i}))\) such that \(\decwidth(\mTotdec(d_i)) \leq \max \{\decwidth(d_i), \card{\image(\partial_i)}\}\).
  Then, \(\mTotdec(d) \defn \nodegenerator{\mTotdec(d_1)}{\image(\partial)}{\mTotdec(d_2)}\) is a tree decomposition of \(\Gamma\) because \(\image(\partial) \union V_1 \union V_2 = V\), \(V_1 \intersection V_2 = \emptyset\) and \(E = E_1 \disjointunion E_2\).

  We compute the width of \(\mTotdec(d)\).
  \begin{align*}
  & \decwidth(\mTotdec(d))\\
  & \defn \max \{\card{\image(\partial)}, \decwidth(\mTotdec(d_1)), \decwidth(\mTotdec(d_2))\} \\
  & \leq \max \{\card{\image(\partial)}, \decwidth(d_1), \decwidth(d_2)\} \\
  & \codefn \max \{\decwidth(d),\card{\image(\partial)}\}
  \end{align*}
\end{proof}

\subsection{Monoidal path width}\label{app:mon-path-width}

\begin{proof}[Proof of~\Cref{prop:pathDecToCatPathDec}]
  Let \(G = \mathgraph{E}{V}\).
  Proceed by induction on \(T\).

  If \(T = (V)\), we can define \(\pTomdec(T) \defn (g)\) and compute its width to be \(\decwidth(\pTomdec(T)) \defn \card{V} \codefn \decwidth(T)\).

  If \(T = (V_1,T')\), then \(V_1 \subseteq V\) and \(T'\) is a recursive path decomposition of \(\Gamma' = (\mathgraph{E'}{V'},X')\) subgraph of \(\Gamma\) such that \(X \subseteq V_1\), \(V_1 \union V' = V\), \(X' = V_1 \intersection V'\) and \(\edgesetends{E \setminus E'} \subseteq V_1\).
  Let \(G' \defn \mathgraph{E'}{V'}\) and \(G_1 \defn \mathgraph{E \setminus E'}{V_1}\).
  Let \(g_1 \defn \cospan{X}{\inclusion}{G_1}{\inclusion}{X'}\) and \(g' \defn \cospan{X'}{\inclusion}{G'}{}{\emptyset}\).
  Then, \(g = g_1 \dcomp g'\) because the pushout of \(\spanmorph{V_1}{\inclusion}{X'}{\inclusion}{V'}\) is \(V_1 \union V' = V\).
  By induction hypothesis, we have \(\pTomdec(T')\) monoidal path decomposition of \(g'\) such that \(\decwidth(\pTomdec(T')) = \decwidth(T')\).
  Then, \(\pTomdec(T) \defn \nodegenerator{g_1}{\dcomp_{X'}}{\pTomdec(T')}\) is a monoidal path decomposition of \(g\):
  \begin{center}
    \mwdPwdUpperProofFig{}
  \end{center}
  We compute the width of \(\pTomdec(T)\).
  \begin{align*}
    & \decwidth(\pTomdec(T))\\
    & \defn \max\{\card{V_1}, \card{X'},\decwidth(\pTomdec(T'))\}\\
    & = \max\{\card{V_1}, \decwidth(\pTomdec(T'))\}\\
    & = \max\{\card{V_1}, \decwidth(T')\}\\
    & \codefn \decwidth(T)
  \end{align*}
\end{proof}

\begin{proof}[Proof of~\Cref{lemma:episPreservePathDec}]
  Let \(G = \mathgraph{E}{V}\) and \(H = \mathgraph{F}{W}\).
  Proceed by induction on \(T\).

  If \(T = (V)\), we can define \(\epiTodec{\alpha}(T) \defn (W)\) and compute its width to be \(\decwidth(\epiTodec{\alpha}(T)) \defn \card{W} \leq \card{V} \codefn \decwidth(T)\) because \(\alpha\) is an epimorphism.

  If \(T = (V_1,T')\), then \(V_1 \subseteq V\) and \(T'\) is a recursive path decomposition of \(G' \subseteq G\) such that \(V_1 \union \vertices(G') = V\) and \(\edgesetends{E \setminus \edges(G')} \subseteq V_1\).
  The epimorphism \(\alpha \colon G \to H\) restricts to an epimorphism \(\alpha' \colon G' \to H'\) with \(H' \subseteq H\).
  It, moreover, satisfies that, if \(\alpha'(v) = \alpha'(w)\), then there is a subtree \(T_{0}\) of \(T'\) with \(v,w \in \labelling'(T_{0})\) because \(\alpha\) satisfies this condition.
  Let \(Y' \defn \alpha(V_{1} \intersection \vertices(G'))\).
  Then, there is a recursive path decomposition \(\epiTodec{\alpha'}(T')\) of \((H',Y')\) such that \(\decwidth(\epiTodec{\alpha'}(T')) \leq \decwidth(T')\) by induction hypothesis.
  Define \(\epiTodec{\alpha}(T) \defn (\alpha(V_1),\epiTodec{\alpha'}(T'))\).
  This is a recursive path decomposition of \((H,Y)\) because we can check the conditions in~\Cref{def:recursivePathDec}:
  \begin{itemize}
    \item \(\alpha(V_1) \union \vertices(H')\) because \(V_1 \union \vertices(G') = V\),
    \item \(\edgesetends{F \setminus \edges(H')} \subseteq(\alpha(V_1))\) because \(\edgesetends{E \setminus \edges(G')} \subseteq(V_1)\),
  \end{itemize}
  by surjectivity of \(\alpha\) on vertices.
  We can compute the width of \(\epiTodec{\alpha}(T)\).
  \begin{align*}
  & \decwidth(\epiTodec{\alpha}(T)) \\
  & \defn \max\{\card{\alpha(V_1)},\decwidth(\epiTodec{\alpha'}(T'))\}\\
  & \leq \max\{\card{\alpha(V_1)},\decwidth(T')\}\\
  & \leq \max\{\card{V_1},\decwidth(T')\}\\
  & \codefn \decwidth(T)
  \end{align*}
\end{proof}

\begin{proof}[Proof of~\Cref{prop:catPathDecToPathDec}]
  Let \(G = \mathgraph{E}{V}\).
  Proceed by induction on \(d\).

  If \(d = \leafgenerator{g}\) then we can define \(\mTopdec(d) \defn \leafgenerator{V}\) and compute its width to be \(\decwidth(\mTopdec(d)) \defn \card{V} \codefn \decwidth(d)\).

  If \(d = \nodegenerator{\leafgenerator{h_{1}}}{\dcomp_{Y}}{d'}\) then \(h_1 = \cospan{X}{}{H_1}{}{Y}\), \(d'\) is a monoidal path decomposition of \(h = \cospan{Y}{\partial_{Y}}{H}{}{\emptyset}\) and \(g = h_1 \dcomp h\).
  By induction hypothesis, we can define \(\mTopdec(d')\) path decomposition of \((H,\image(\partial_{Y}))\).
  By \Cref{lemma:episFromComposition}, there are subgraphs \(G_1\) and \(G'\) of \(G\) and graph epimorphisms \(\alpha_1 \colon H_1 \to G_1\) and \(\alpha \colon H \to G'\).
  All the vertices of \(H\) that are identified by \(\alpha\) need to be in \(\image(\partial_{Y})\).
  Thus, \(\alpha\) satisfies the hypothesis of~\Cref{lemma:episPreservePathDec}.
  We can, then, apply this lemma to obtain a recursive path decomposition \(\epiTodec{\alpha}(\mTopdec(d'))\) of \((G',\image(\partial_{Y} \dcomp \alpha))\) such that \(\decwidth(\epiTodec{\alpha}(\mTopdec(d'))) \leq \decwidth(\mTopdec(d'))\).
  Let \(V_1 \defn \vertices(G_1)\), \(E_1 \defn \edges(G_1)\), \(V' \defn \vertices(G')\) and \(E' \defn \edges(G')\).
  We define \(\mTopdec(d) \defn (V_1,\epiTodec{\alpha}(\mTopdec(d')))\).
  This is a recursive path decomposition of \(\Gamma\) because we can check the conditions of \Cref{def:recursivePathDec}.
  \begin{itemize}
    \item \(\image(\partial) \subseteq V_1\) because it factors through \(\vertices(H_1)\),
    \item \(V_1 \union V' = V\) because of composition by pushout,
    \item \(X' \defn V_1 \intersection V'\),
    \item \(\edgesetends{E\setminus E'} = \edgesetends{E_1} \subseteq V_1\) because of composition by pushout.
  \end{itemize}

  We can compute the width of \(\mTopdec(d)\):
  \begin{align*}
    & \decwidth(\mTopdec(d))\\
    & \defn \max\{\card{V_1},\decwidth(\epiTodec{\alpha}(\mTopdec(d')))\}\\
    & \leq \max\{\card{V_1},\decwidth(\mTopdec(d'))\}\\
    & \leq \max\{\card{V_1},\decwidth(d')\}\\
    & \leq \max\{\card{V_1},\card{Y},\decwidth(d')\}\\
    & \codefn \decwidth(d)
  \end{align*}
\end{proof}

\subsection{Monoidal width}\label{app:mon-branch-width}

\begin{proof}[Proof of~\Cref{prop:mwd-branch-width-upper-bound}]
  Let \(G = \mathgraph{E}{V}\) and proceed by induction on \(\treedepth(T)\).

  If \(T = (\Gamma)\), then \(\bTomdec(T) \defn (g)\) and \(\decwidth(T) \defn 1 \codefn \decwidth(\bTomdec(T))\).

  If \(T = \nodegenerator{T_1}{\Gamma}{T_2}\), then, by definition of branch decomposition, \(T_i\) is a branch decomposition of \(\Gamma_i = (G_i,X_i)\) with \(G_i = \mathgraph{E_i}{V_i}\), \(E = E_1 \disjointunion E_2\) with \(E_i \neq \emptyset\), \(V_1 \union V_2 = V\), and \(X_i = (V_1 \intersection V_2) \union (X \intersection V_i)\).
  Let \(g_i = \cospan{X_i}{\inclusion}{G_i}{}{\emptyset}\).
  Then,
  \begin{center}
    \bwdProofFig{}
  \end{center}
  and we can write \(g = (\id{X \intersection X_1} \tensor \codelete_{X_1 \intersection X_2} \tensor \id{X_2 \setminus X_1}) \dcomp (\id{X_1 \setminus X_2} \tensor \cp_{X_1 \intersection X_2} \tensor \id{X_2 \setminus X_1}) \dcomp (g_1 \tensor \id{X_2}) \dcomp g_2\).
  By induction, there are \(\bTomdec(T_i)\) monoidal decompositions of \(g_i\) such that \(\decwidth(\bTomdec(T_i)) \leq \decwidth(T_i) + 1\).
  By \Cref{lemma:mwd-copy}, there is a monoidal decomposition \(\copyMdec(\bTomdec(T_1))\) of \((\id{X \intersection X_1} \tensor \codelete_{X_1 \intersection X_2} \tensor \id{X_2 \setminus X_1}) \dcomp (\id{X_1 \setminus X_2} \tensor \cp_{X_1 \intersection X_2} \tensor \id{X_2 \setminus X_1}) \dcomp (g_1 \tensor \id{X_2})\) such that \(\decwidth(\copyMdec(\bTomdec(T_1))) \leq \max\{\decwidth(\bTomdec(T_1)), \card{X_1}+1\}\).
  Define
  \[\bTomdec(T) \defn \nodegenerator{\nodegenerator{\copyMdec(\bTomdec(T_1))}{\tensor}{\id{X_{2}\setminus X_{1}}}}{\dcomp_{X_2}}{\bTomdec(T_2)}.\]
  Then,
  \begin{align*}
  & \decwidth(\bTomdec(T))\\
  & \defn \max \{\decwidth(\copyMdec(\bTomdec(T_1))), \decwidth(\bTomdec(T_2)), \card{X_2}\}\\
  & = \max \{\decwidth(\bTomdec(T_1)), \decwidth(\bTomdec(T_2)), \card{X_1} + 1, \card{X_2}\}\\
  & \leq \max \{\decwidth(T_1) + 1, \decwidth(T_2) + 1, \card{X_1} + 1, \card{X_2}\}\\
  & \leq \max \{\decwidth(T_1), \decwidth(T_2), \card{X_1}, \card{X_2}\} + 1\\
  & \leq \max \{\decwidth(T_1), \decwidth(T_2), \card{X}\} + 1\\
  & \codefn \decwidth(T) +1
  \end{align*}
  because \(\card{X_i} \leq \decwidth(T_i)\) by definition.
\end{proof}

\begin{remark}\label{rem:compute-images}
  Let \(f \colon A \to C\) and \(g \colon B \to C\) be two functions.
  The union of the images of \(f\) and \(g\) is the image of the coproduct map \(\coproductmap{f}{g} \colon A + B \to C\), i.e. \(\image(f) \union \image(g) = \image(\coproductmap{f}{g})\).
  The intersection of the images of \(f\) and \(g\) is the image of the pullback map \(\pullbackmap{f}{g} \colon A \times_{C} B \to C\), i.e. \(\image(f) \intersection \image(g) = \image(\pullbackmap{f}{g})\).
\end{remark}

\begin{remark}\label{rem:glueing-property}
  Let \(f \colon A \to C\), \(g \colon B \to C\) and \(\phi \colon C \to V\) such that \(\forall \ c \neq c' \in C \ \phi(c) = \phi(c') \implies c,c' \in \image(f)\).
  Then, \(\image(\pullbackmap{f \dcomp \phi}{g \dcomp \phi}) = \image(\pullbackmap{f}{g} \dcomp \phi)\) because
  \begin{align*}
    & \image(\pullbackmap{f}{g} \dcomp \phi) \setminus \image(\pullbackmap{f \dcomp \phi}{g \dcomp \phi}) \\
    & = \{v \in V : \exists a \in A \ \exists b \in B \ \phi(f(a)) = \phi(g(b)) \land f(a) \notin \image(g) \land g(b) \notin \image(f)\}\\
    & = \emptyset
  \end{align*}
\end{remark}

\begin{proof}[Proof of~\Cref{prop:mwd-branch-width-lower-bound}]
  By induction on \(d\).

  If \(d = \leafgenerator{h}\) and \(F = \emptyset\), then \(\mTobdec(d) \defn \emptydec\) and we can compute \(\decwidth(\mTobdec(d)) \defn 0 \leq 2 \cdot \max \{\decwidth(d), \card{A}, \card{B}\}\).

  If \(d = \leafgenerator{h}\) and \(F = \{e\}\), then \(\mTobdec(d) \defn \leafgenerator{\Gamma}\) and we can compute \(\decwidth(\mTobdec(d)) \defn \card{\image(\partial_A \dcomp \phi) \union \image(\partial_B \dcomp \phi)} \leq \card{A} + \card{B} \leq 2 \cdot \max \{\decwidth(d), \card{A}, \card{B}\}\).

  If \(d = \leafgenerator{h}\) and \(\card{F} > 1\), then let \(\mTobdec(d)\) be any recursive branch decomposition of \(\Gamma\).
  Its width is not greater than the number of vertices in \(\Gamma\), thus \(\decwidth(\mTobdec(d)) \leq \card{\image(\phi)} \leq 2 \cdot \max\{\decwidth(d), \card{A}, \card{B}\}\).

  If \(d = \nodegenerator{d_1}{\dcomp_C}{d_2}\), then \(d_i\) is a monoidal decomposition of \(h_i\) with \(h = h_1 \dcomp_C h_2\).
  Let \(h_1 = \cospan{A}{\partial^1_A}{H_1}{\partial_1}{C}\) and \(h_2 = \cospan{C}{\partial_2}{H_2}{\partial^2_B}{B}\) with \(H_i = \mathgraph{F_i}{W_i}\).
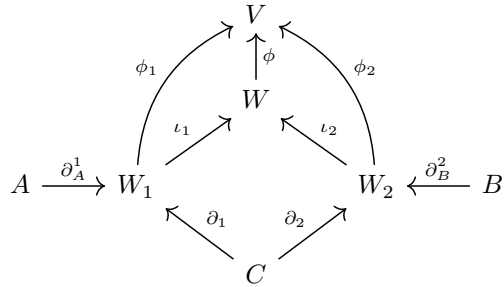
  \begin{center}
    \begin{tikzcd}
      & & {V} & & \\
      & & {W} \arrow[u, "{\phi}"'] & & \\
      {A} \arrow[r, "{\partial_A^1}"] & {W_{1}} \arrow[ru, "{\inclusion[1]}"] \arrow[ruu, "{\phi_1}", bend left] & & {W_{2}} \arrow[lu, "{\inclusion[2]}"'] \arrow[luu, "{\phi_2}"', bend right] & {B} \arrow[l, "{\partial_B^2}"']\\
      & & {C} \arrow[lu, "{\partial_1}"'] \arrow[ru, "{\partial_2}"] & &
    \end{tikzcd}
  \end{center}
  Let \(\inclusion[i] \colon W_i \to W\) be the functions induced by the pushout.
  Define \(\phi_i \defn \inclusion[i] \dcomp \phi\).
  We show that \(\phi_1\) satisfies the glueing property:
  let \(w \neq w' \in W_1\) such that \(\phi_1(w) = \phi_1(w')\).
  Then, \(\inclusion[1](w) = \inclusion[1](w')\) or \(\phi(\inclusion[1](w)) = \phi(\inclusion[1](w')) \land \inclusion[1](w) \neq \inclusion[1](w')\).
  Then, \(w,w' \in \image(\partial_1)\) or \(\inclusion[1](w), \inclusion[1](w') \in \image(\partial_A \dcomp \phi) \union \image(\partial_B \dcomp \phi)\).
  Then, \(w,w' \in \image(\partial_1)\) or \(w,w' \in \image(\partial^1_A)\).
  Then, \(w,w' \in \image(\partial_1) \union \image(\partial^1_A)\).
  Similarly, we can show that \(\phi_2\) satisfies the same property.
  Then, we can apply the induction hypothesis to get a recursive branch decomposition \(\mTobdec(d_1)\) of \(\Gamma_1 = (\mathgraph{F_1}{\image(\phi_1)}, \image(\partial^1_A \dcomp \phi_1) \union \image(\partial_1 \dcomp \phi_1))\) and a recursive branch decomposition \(\mTobdec(d_2)\) of \(\Gamma_2 = (\mathgraph{F_2}{\image(\phi_2)}, \image(\partial^2_B \dcomp \phi_2) \union \image(\partial_2 \dcomp \phi_2))\) such that \(\decwidth(\mTobdec(d_1)) \leq 2 \cdot \max\{\decwidth(d_1),\card{A},\card{C}\}\) and  \(\decwidth(\mTobdec(d_2)) \leq 2 \cdot \max\{\decwidth(d_2),\card{B},\card{C}\}\).

  We check that we can define a recursive branch decomposition of \(\Gamma\) from \(\mTobdec(d_1)\) and \(\mTobdec(d_2)\).
  \begin{itemize}
    \item \(F = F_1 \disjointunion F_2\) because the pushout is along discrete graphs.
    \item \(\image(\phi) = \image(\phi_1) \union \image(\phi_2)\) because \(\image(\coproductmap{\inclusion[1]}{\inclusion[2]}) = W\) and \(\image(\phi_1) \union \image(\phi_2) = \image(\inclusion[1] \dcomp \phi) \union \image(\inclusion[2] \dcomp \phi) = \image(\coproductmap{\inclusion[1]}{\inclusion[2]} \dcomp \phi) = \image(\phi)\).
    \item \(\image(\coproductmap{\partial^1_A}{\partial_1} \dcomp \phi_1) = \image(\phi_1) \intersection (\image(\phi_2) \union \image(\partial_A \dcomp \phi) \union \image(\partial_B \dcomp \phi))\) because
     \begin{align*}
       & \image(\phi_1) \intersection (\image(\phi_2) \union \image(\partial_A \dcomp \phi) \union \image(\partial_B \dcomp \phi)) \\
       \explain{by definition of \(\phi_i\)}\\
       & \image(\inclusion[1] \dcomp \phi) \intersection (\image(\inclusion[2] \dcomp \phi) \union \image(\partial_A \dcomp \phi) \union \image(\partial_B \dcomp \phi)) \\
       \explain{because \(\image(\partial_B) = \image(\partial^2_B \dcomp \inclusion[2]) \subseteq \image(\inclusion[2])\)}\\
       & \image(\inclusion[1] \dcomp \phi) \intersection (\image(\inclusion[2] \dcomp \phi) \union \image(\partial_A \dcomp \phi)) \\
       \explain{by~\Cref{rem:compute-images}}\\
       & \image(\inclusion[1] \dcomp \phi) \intersection \image(\coproductmap{\inclusion[2]}{\partial_A} \dcomp \phi) \\
       \explain{by~\Cref{rem:compute-images}}\\
       & \image(\pullbackmap{\inclusion[1] \dcomp \phi}{\coproductmap{\inclusion[2]}{\partial_A} \dcomp \phi})\\
       \explain{by~\Cref{rem:glueing-property}}\\
       & \image(\pullbackmap{\inclusion[1]}{\coproductmap{\inclusion[2]}{\partial_A}} \dcomp \phi)\\
       \explain{because pullbacks commute with coproducts}\\
       & \image(\coproductmap{\pullbackmap{\inclusion[1]}{\inclusion[2]}}{\pullbackmap{\inclusion[1]}{\partial_A}} \dcomp \phi)\\
       \explain{because \(\partial_A = \partial^1_A \dcomp \inclusion[1]\)}\\
       & \image(\coproductmap{\pullbackmap{\inclusion[1]}{\inclusion[2]}}{\partial_A} \dcomp \phi)\\
       \explain{because \(\partial_1 \dcomp \inclusion[1] = \partial_2 \dcomp \inclusion[2]\) is the pushout map of \(\partial_1\) and \(\partial_2\)}\\
       & \image(\coproductmap{\partial_1 \dcomp \inclusion[1]}{\partial^1_A \dcomp \inclusion[1]} \dcomp \phi)\\
       \explain{by property of the coproduct}\\
       & \image(\coproductmap{\partial_1}{\partial^1_A} \dcomp \phi_1)
     \end{align*}
    \item \(\image(\coproductmap{\partial_2}{\partial^2_B} \dcomp \phi_2) = \image(\phi_2) \intersection (\image(\phi_1) \union \image(\partial_A \dcomp \phi) \union \image(\partial_B \dcomp \phi))\) similarly to the former point.
  \end{itemize}
  Then, \(\mTobdec(d) \defn \nodegenerator{\mTobdec(d_1)}{\Gamma}{\mTobdec(d_2)}\) is a recursive branch decomposition of \(\Gamma\) and
  \begin{align*}
    & \decwidth(\mTobdec(d)) \\
    & \defn \max\{\decwidth(\mTobdec(d_1)), \card{\image(\coproductmap{\partial_A}{\partial_B})}, \decwidth(\mTobdec(d_2)) \}\\
    & \leq \max\{2 \cdot \decwidth(d_1), 2 \cdot \card{A}, 2 \cdot \card{C}, \card{A} + \card{B}, 2 \cdot \decwidth(d_2), 2 \cdot \card{B} \}\\
    & \leq 2 \cdot \max\{\decwidth(d_1), \card{A}, \card{C}, \decwidth(d_2), \card{B} \}\\
    & \codefn 2 \cdot \max\{\decwidth(d), \card{A}, \card{B} \}
  \end{align*}

  If \(d = \nodegenerator{d_1}{\tensor}{d_2}\), then \(d_i\) is a monoidal decomposition of \(h_i\) with \(h = h_1 \tensor h_2\).
  Let \(h_i = \cospan{X_i}{\partial^i_X}{H_i}{\partial^i_Y}{Y_i}\) with \(H_i = \ctgraph{F_i}{W_i}\).
  Let \(\inclusion[i] \colon W_i \to W\) be the inclusions induced by the monoidal product.
  Define \(\phi_i \defn \inclusion[i] \dcomp \phi\).
  We show that \(\phi_1\) satisfies the glueing property:
  Let \(w \neq w' \in W_1\) such that \(\phi_1(w) = \phi_1(w')\).
  Then, \(\inclusion[1](w) = \inclusion[1](w')\) or \(\phi(\inclusion[1](w)) = \phi(\inclusion[1](w')) \land \inclusion[1](w) \neq \inclusion[1](w')\).
  Then, \(\inclusion[1](w), \inclusion[1](w') \in  \image(\partial_A \dcomp \phi) \union \image(\partial_B \dcomp \phi)\) because \(\inclusion[i]\) are injective.
  Then, \(w,w' \in \image(\partial^1_A) \union \image(\partial^1_B)\).
  Similarly, we can show that \(\phi_2\) satisfies the same property.
  Then, we can apply the induction hypothesis to get \(\mTobdec(d_i)\) recursive branch decomposition of \(\Gamma_i = (\mathgraph{F_i}{\image(\phi_i)}, \image(\coproductmap{\partial^i_A}{\partial^i_B} \dcomp \phi_i))\) such that \(\decwidth(\mTobdec(d_i)) \leq 2 \cdot \max\{\decwidth(d_i),\card{A_i},\card{B_i}\}\).

  We check that we can define a recursive branch decomposition of \(\Gamma\) from \(\mTobdec(d_1)\) and \(\mTobdec(d_2)\).
  \begin{itemize}
    \item \(F = F_1 \disjointunion F_2\) because the monoidal product is given by the coproduct in \(\Set\).
    \item \(\image(\phi) = \image(\phi_1) \union \image(\phi_2)\) because \(\image(\coproductmap{\inclusion[1]}{\inclusion[2]}) = W\) and \(\image(\phi_1) \union \image(\phi_2) = \image(\inclusion[1] \dcomp \phi) \union \image(\inclusion[2] \dcomp \phi) = \image(\coproductmap{\inclusion[1]}{\inclusion[2]} \dcomp \phi) = \image(\phi)\).
    \item \(\image(\coproductmap{\partial^1_A}{\partial^1_B} \dcomp \phi_1) = \image(\phi_1) \intersection (\image(\phi_2) \union \image(\partial_A \dcomp \phi) \union \image(\partial_B \dcomp \phi))\) because
     \begin{align*}
       & \image(\phi_1) \intersection (\image(\phi_2) \union \image(\partial_A \dcomp \phi) \union \image(\partial_B \dcomp \phi)) \\
       \explain{by definition of \(\phi_i\)}\\
       & \image(\inclusion[1] \dcomp \phi) \intersection (\image(\inclusion[2] \dcomp \phi) \union \image(\partial_A \dcomp \phi) \union \image(\partial_B \dcomp \phi)) \\
       \explain{by~\Cref{rem:compute-images} and property of the coproduct}\\
       & \image(\inclusion[1] \dcomp \phi) \intersection \image(\coproductmap{\inclusion[2]}{\coproductmap{\partial_A}{\partial_B}} \dcomp \phi)\\
       \explain{by~\Cref{rem:compute-images}}\\
       & \image(\pullbackmap{\inclusion[1] \dcomp \phi}{\coproductmap{\inclusion[2]}{\coproductmap{\partial_A}{\partial_B}} \dcomp \phi})\\
       \explain{by~\Cref{rem:glueing-property}}\\
       & \image(\pullbackmap{\inclusion[1]}{\coproductmap{\inclusion[2]}{\coproductmap{\partial_A}{\partial_B}}} \dcomp \phi)\\
       \explain{because pullbacks commute with coproducts}\\
       & \image(\coproductmap{\pullbackmap{\inclusion[1]}{\inclusion[2]}}{\pullbackmap{\inclusion[1]}{\coproductmap{\partial_A}{\partial_B}}} \dcomp \phi)\\
       \explain{because \(\pullbackmap{\inclusion[1]}{\inclusion[2]} = \initmap{}\)}\\
       & \image(\pullbackmap{\inclusion[1]}{\coproductmap{\partial_A}{\partial_B}} \dcomp \phi)\\
       \explain{because \(\partial_A = \partial^1_A + \partial^2_A\) and \(\partial_B = \partial^1_B + \partial^2_B\)}\\
       & \image(\coproductmap{\partial^1_A \dcomp \inclusion[1]}{\partial^1_B \dcomp \inclusion[1]} \dcomp \phi)\\
       \explain{by property of the coproduct}\\
       & \image(\coproductmap{\partial^1_A}{\partial^1_B} \dcomp \phi_1)
     \end{align*}
    \item \(\image(\coproductmap{\partial^2_A}{\partial^2_B} \dcomp \phi_2) = \image(\phi_2) \intersection (\image(\phi_1) \union \image(\partial_A \dcomp \phi) \union \image(\partial_B \dcomp \phi))\) similarly to the former point.
  \end{itemize}
  Then, \(\mTobdec(d) \defn \nodegenerator{\mTobdec(d_1)}{\Gamma}{\mTobdec(d_2)}\) is a recursive branch decomposition of \(\Gamma\) and
  \begin{align*}
    & \decwidth(\mTobdec(d)) \\
    & \defn \max\{\decwidth(\mTobdec(d_1)), \card{\image(\coproductmap{\partial_A}{\partial_B})}, \decwidth(\mTobdec(d_2)) \}\\
    & \leq \max\{2 \cdot \decwidth(d_1), 2 \cdot \card{A_1}, 2 \cdot \card{B_1}, \card{A} + \card{B}, 2 \cdot \decwidth(d_2), 2 \cdot \card{A_2}, 2 \cdot \card{B_2} \}\\
    & \leq 2 \cdot \max\{\decwidth(d_1), \card{A}, \decwidth(d_2), \card{B} \}\\
    & \codefn 2 \cdot \max\{\decwidth(d), \card{A}, \card{B} \}
  \end{align*}
\end{proof}

\end{document}